\documentclass[12pt,reqno]{amsart}
\usepackage{amsthm,amsfonts, amssymb, amscd, mathrsfs, hyperref}
\newtheorem{thm}{Theorem}
\newtheorem{lemma}[thm]{Lemma}
\newtheorem{prop}[thm]{Proposition}
\newtheorem{corol}[thm]{Corollary}

\theoremstyle{definition}
\newtheorem{defi}[thm]{\textbf{Definition}}
\newtheorem{remark}[thm]{\textbf{Remark}}

\newtheorem{example}[thm]{\textbf{Example}}
\numberwithin{subsection}{section}

\newcommand{\C}{\mathbb C}

\newcommand{\Q}{\mathbb Q}
\newcommand{\Z}{\mathbb Z}

\newcommand{\Fc}{\mathscr F}
\newcommand{\Ic}{\mathfrak I}  
\newcommand{\IX}{\Ic_\X}
\newcommand{\IIX}{{\Ic \kern-.75em \Ic}_{\X}}
\newcommand{\II}{\Itwo}

\newcommand{\X}{\mathscr X}
\newcommand{\T}{\mathbb T}
\newcommand{\cho}{\mathfrak{ch}}

\DeclareMathOperator{\GL}{GL}

\DeclareMathOperator{\Spec}{Spec}

\DeclareMathOperator{\ch}{ch}
\DeclareMathOperator{\Td}{Td}

\DeclareMathOperator{\Stab}{Stab}

\newcommand{\bg}{\mathbf{g}}
\newcommand{\bm}{\mathbf{m}}
\newcommand{\ce}{\mathscr{E}}

\newcommand{\G}[1]{G^{[#1]}}

\newcommand{\Supp}{\operatorname{Supp}}
\newcommand{\Rep}{\operatorname{Rep}} 
\newcommand{\Lie}{\operatorname{Lie}}
\newcommand{\kclass}{\mathscr E} 
\newcommand{\Imult}[1]{\mathbb{I}^{#1}}  
\newcommand{\Itwo}{\Imult{2}}
\newcommand{\euler}{\varepsilon}
\newcommand{\musub}[1]{\mu_{#1}}  
\newcommand{\mui}{\musub{i}}
\newcommand{\esub}[1]{\mathbf{e}_{#1}}
\newcommand{\ej}{\esub{j}}
\begin{document}
\numberwithin{thm}{section}
\title{Logarithmic trace and orbifold products}
\author{Dan Edidin, Tyler J. Jarvis, Takashi Kimura}
\address{Department of Mathematics, University of Missouri, Columbia, MO 65211}
\email{edidind@missouri.edu}
\address{Department of Mathematics, Brigham Young University, Provo, UT 84602}
\email{jarvis@math.byu.edu}
\address{Department of Mathematics and Statistics, 111 Cummington Street, Boston University, Boston, MA 02215} 
\email{kimura@math.bu.edu}
\maketitle
\date{\today}
\begin{abstract}
The purpose of this paper is to give a purely equivariant definition
of orbifold Chow rings of quotient Deligne-Mumford stacks. This
completes a program begun in \cite{JKK:07} for quotients by finite
groups. The key to our construction is the definition (Section
\ref{subsec.twistedpullback}), of a twisted pullback in equivariant
$K$-theory, $K_G(X) \to K_G( \II_G(X))$ taking non-negative
elements to non-negative elements.  (Here 
$\II_G(X) = \{(g_1,g_2,x)|g_1x = g_2 x = x \} \subset G \times G \times X$.) 
The
twisted pullback is defined using data about fixed loci of elements
of finite order in $G$, but depends only on the underlying quotient
stack (Theorem~\ref{thm.twcanonical}). In our theory, the twisted
pullback of the class $\T \in K_G(X)$, corresponding to the tangent
bundle to $[X/G]$, replaces the obstruction bundle of the corresponding
moduli space of twisted stable maps. When $G$ is finite, the twisted
pullback of the tangent bundle agrees with the class $R({\mathbf
m})$ given in \cite[Definition 1.5]{JKK:07}. However, unlike in
\cite{JKK:07} we 
need not
compare our class to the 
class of the obstruction
bundle of Fantechi and G\"ottsche \cite{FaGo:03}
in order 
to prove that it is
a non-negative integral element of $K_G(\II_G(X))$.

We also give an equivariant description of the product on the orbifold
$K$-theory of $[X/G]$.  Our orbifold Riemann-Roch theorem
(Theorem~\ref{thm.orbrr}) states that there is an orbifold Chern
character homomorphism which induces an isomorphism of a canonical
summand in the orbifold Grothendieck ring with the orbifold Chow ring.
As an application we show (Theorem~\ref{thm.orbonkt}) that if $\X =
[X/G]$, then there is an associative orbifold product structure on
$K(\X)\otimes \C$ distinct from the usual tensor product.
\end{abstract}

\tableofcontents

\section{Introduction}
\subsection{Statement of results}
In this paper we work in the algebraic category and consider quasi-free
actions of arbitrary algebraic groups on arbitrary smooth
varieties (or more generally algebraic spaces). In practice most
smooth, separated Deligne-Mumford stacks have natural presentations as
quotients of the form $[X/G]$ with $X$ smooth and $G$ an
algebraic group acting properly on $X$. (The paper
\cite{EHKV:01} gives criteria for a Deligne-Mumford stack to be a
quotient stack. There are in fact no known examples of separated
Deligne-Mumford stacks which are provably not quotient stacks.) 

Our first main result is a purely equivariant description of the
orbifold product on the Chow groups of the inertia stack $\IX =
[I_G(X)/G]$, where
$$I_G(X)=\{(g,x)| gx =x\} \subseteq  G \times X.$$ 
The product depends only on data about fixed loci of elements of
finite order and makes no reference to moduli spaces and obstruction
bundles. This completes a program begun in \cite{JKK:07}.  In
particular we show, without use of the character formula of
\cite[Lemma 8.5]{JKK:07}, that when $G$ is finite the class
$R({\bm})$ given in \cite[Definition 1.5]{JKK:07} is a
non-negative integral element of $K$-theory. To do this we define
(Section~\ref{subsec.twistedpullback}), for arbitrary $G$ acting
quasi-freely on a smooth variety $X$, an explicit \emph{twisted pullback}
map on integral equivariant $K$-theory $K_G(X) \to K_G(\II_G(X))$,
where $\II_G(X) = I_G(X) \times_X I_G(X)$. The twisted pullback takes
non-negative elements to non-negative elements and depends only on the
underlying quotient stack $\X= [X/G]$---not on the 
specific choice of presentation 
$X$ and $G$. The twisted pullback we define for quotient  Deligne-Mumford
stacks is a special case of a more general construction given in Sections~\ref{sec.logtrace} and~\ref{sec.logres}. For the reader's
convenience we outline the construction in Section
\ref{subsec.introtwisted} below.

When $G$ is finite, the class $R({\bm})$ of \cite{JKK:07} is a
component of the twisted pullback of the tangent bundle $TX$. More
generally, for $G$ acting with finite stabilizer on $X$ we define a
product on the equivariant Chow groups $A^*_G(I_G(X))$ by the formula
\begin{equation*}
\alpha \star_{c_\T} \beta = \mu_*\left(
e_1^*\alpha \cup e_2^*\alpha \cup \euler(\T^{tw})
\right).
\end{equation*}
Here $\mu \colon \II_G(X) \to I_G(X)$ is 
induced from the group multiplication $G\times G\to G$,
the map $e_i$ is the projection onto the $i$-th factor, 
$\T^{tw}$ is the twisted pullback of
the class $\T \in K_G(X)$ corresponding to the tangent bundle of
$[X/G]$,
and $\euler(\T^{tw})$ denotes the Euler class.
In Theorem~\ref{thm.assoc}, we prove that the
$\star_{c_\T}$ product is commutative and associative by showing that
it satisfies the sufficiency conditions (Propositions
\ref{prop.chowidentity},\ref{prop.commutative},\ref{prop.assoc}) for an
\emph{inertial product} to be commutative and associative.
Example~\ref{exam.otherassoc} gives
an example of a different associative product on $A^*_G(I_G(X))$. An
interesting question is to classify all possible associative products
on the Chow groups $A^*_G(I_G(X))$.

The $\star_{c_\T}$ product can be explicitly calculated using the
decomposition of $I_G(X)$ into a disjoint sum of components
$\coprod_\Psi I(\Psi)$, where the sum runs over all conjugacy
classes $\Psi \subset G$ of elements of finite order. Here $I(\Psi) =
\{(g,x)| gx=x, g \in \Psi\}$. Note that this disjoint sum is finite,
as $I(\Psi) = \emptyset$ for all but finitely many conjugacy classes
$\Psi$.  The equivariant Chow groups $A^*_G(I(\Psi))$ may be
identified with $A^*_Z(X^g)$, where $g \in \Psi$ is any element and $Z
= Z_G(g)$ is the centralizer of $g$ in $G$. In this way the
$\star_{c_\T}$ product can be computed purely in terms of the fixed
point data of elements of finite order for the action of $X$ on $G$.

When $X$ is a smooth scheme an analogous definition can be made in
equivariant $K$-theory, replacing the 
Euler class of $\T^{tw}$ in the equivariant Chow group with the K-theoretic Euler class $\lambda_{-1}((\T^{tw})^*)$.
In this case we define
an \emph{orbifold Chern character} 
$$\cho \colon K_G(I_G(X)) \otimes \Q \to A^*_G(I_G(X))\otimes \Q$$ 
and prove (Theorem~\ref{thm.orbrr}) that it
preserves the corresponding
orbifold products.  The orbifold Chern character is not an isomorphism
but it restricts to an isomorphism  on a
summand in $K_G(I_G(X)) \otimes \Q$ which depends only on the stack
$[X/G]$. When $G$ is finite this summand equals the small orbifold
$K$-theory defined in \cite{JKK:07}.

Finally in Section~\ref{sec.orbonktheory},
we combine the orbifold Riemann-Roch theorem with the
non-Abelian localization theorem of \cite{EdGr:05} to obtain a
twisted, or orbifold, product on $K_G(X) \otimes \C$. Again this
product depends only on the underlying quotient stack $[X/G]$ and not
on the particular presentation.

\subsection{Twisted pullbacks}  \label{subsec.introtwisted}
Let $G$ be an algebraic
group acting on a space $X$ with arbitrary stabilizers. Let ${\bf m} =
(m_1, \ldots , m_l)$ be an $l$-tuple of elements 
in $G$
(not necessarily of
finite order) which lie in a compact subgroup $K \subset G$ and
satisfy $\prod_{i=1}^l m_i = 1$. And let 
$Z = Z_G(\bm) = \bigcap_{i=1}^l Z_G(m_i)$.
In Section
\ref{subsec.logres}, we define a map 
$K_G(X) 
\to K_Z(X^{\mathbf  m})$, called the \emph{logarithmic restriction},
where $X^\bm$ consists of the subset of $X$ consisting of points fixed by $m_i$ for all $i=1,\ldots,l$.
This map takes non-negative elements to 
non-negative
elements and can be used to define (Section
\ref{subsec.twistpullback}) a twisted pullback map 
$K_G(X)\to K_G(\Imult{l}(\Phi(\mathbf{m}))$, 
where $\Imult{l}(\Phi(\mathbf{m}))$ is the set of pairs $(\mathbf{g}, x) \subset G^l \times X$ such that
$\mathbf{g} = (g_1, \ldots , g_l)$ is conjugate to $\mathbf{m}$ under the diagonal conjugation action of $G$ on $G^l$
and $x$ is fixed by each $g_i$ for $i= 1, \ldots , l$. 
When $G$ acts 
quasi-freely,
the twisted pullback 
$K_G(X) \to K_G(\II_G(X))$ 
is
defined using the decomposition of $\II_G(X)$ into open and closed components
indexed by diagonal conjugacy classes in $G \times G$.  

In a subsequent paper, we plan to use the general twisted pullback construction
to define ``stack products'' for Artin quotient stacks.

The definition of the logarithmic restriction, and hence the twisted
pullback, is based on a $K$-theoretic version of an inequality for the
arguments of the eigenvalues of unitary matrices (Section~\ref{sec.logtrace}).
Precisely, if $g$ is a unitary matrix of rank $n$ we define the {\em
  logarithmic trace} $L(g)$ by the formula $L(g) =
\sum_{i=1}^{l}\alpha_i$, where 
$0 \leq \alpha_i < 1$
and
$\{\exp(2\pi\sqrt{-1}\alpha_i)\}_{i=1}^n$ are the eigenvalues of $g$.
The logarithmic trace can be extended to equivariant $K$-theory as
follows.   Suppose that $Y$ is a space with the action of an algebraic
group $Z$ and $V$ is a $Z$-equivariant bundle on $Y$. If $g \in U(n)$
acts on the fibers of $V \to Y$ 
and commutes with the action of $Z$,
then $V$ decomposes into $g$-eigenbundles, each of which is a $Z$-equivariant
vector bundle. As a result we may define
the logarithmic trace $L(g)(V)$ as
an element of 
$K_Z(Y)\otimes {\mathbb R}$.  A key fact, proved by Falbel and Wentworth
\cite{FaWe:06}, states that if $g_1,\ldots ,g_l$ are unitary matrices
satisfying 
$\prod_{i=1}^lg_i =1$ 
then $\sum_{i=1}^l L(g_i)\geq n - n_0$,
where $n_0$ is the 
dimension of the 
subspace fixed by 
$g_i$ for all $i\in\{1,\ldots,l\}$.
Applying this to
equivariant $K$-theory implies that if $g_1,\ldots ,g_l$ 
all act on the fibers of $V \to Y$, 
then $\sum_{i=1}^l L(g_i)(V) - V + V^{\bg}$
is a non-negative (integral) element in $K_Z(Y)$.  
When $Y = X^{{\bg}}$ and 
$Z = Z_G(\bg)$ for $\bg = (g_1,\ldots,g_n)$,
we obtain the logarithmic restriction map discussed above.

\subsection{Connection to orbifold cohomology and other literature}
Orbifold cohomology was originally defined by Chen and Ruan in their
landmark paper \cite{ChRu:04}. They showed that there is a
$\Q$-graded,
associative, super-commutative product structure on the
cohomology groups of the inertia orbifold $\IX$ associated to an
orbifold $\X$.  The orbifold cohomology ring is the degree-zero
part of the quantum cohomology of the orbifold $\X$, and the
product is defined via integration against the virtual fundamental
class of the moduli stack of ghost maps.

Subsequently there has been a great deal of interest in the orbifold
product. Simpler descriptions of orbifold
cohomology have been given for global quotient stacks, that is stacks
of the form $[X/G]$ with $X$ a manifold and $G$ a finite group,
\cite{FaGo:03, JKK:07}. The theory has also been extended to Chow
groups \cite{AGV:08, JKK:07} and $K$-theory \cite{JKK:07}. Borisov,
Chen and Smith calculated the orbifold Chow rings of toric stacks
\cite{BCS:05} and in symplectic geometry an orbifold cohomology for
torus actions was computed by Goldin, Holm and Knutson in
\cite{GHK:07}.

The orbifold Chow and $K$-theory rings we define here extend earlier
definitions of \cite{JKK:07} (and implicitly \cite{FaGo:03}) given for
actions of finite groups.  In this paper, we do not work with
equivariant cohomology, but the formalism we develop works equally
well for actions of compact Lie groups on almost complex
manifolds. The character formula of \cite[Lemma 8.5]{JKK:07} implies
that our product on equivariant cohomology of the inertia group scheme
agrees, after tensoring with $\Q$, with that defined by Fantechi and
G\"ottsche in \cite{FaGo:03}. However in both \cite{FaGo:03} and
\cite{JKK:07}, the orbifold product is defined on the $G$-invariant
part of the stringy cohomology (resp. Chow groups) of $X$. These groups
are rationally isomorphic to the equivariant cohomology of $I_G(X)$, but in
general they are a coarser invariant than equivariant cohomology. 
For example, if $G= \mu_n$ and $X = \Spec \C$, then the additive structure
on Fantechi and G\"ottsche's orbifold cohomology is the Abelian group
$\Z^n$, while the $\mu_n$ equivariant cohomology of $I_{\mu_n}(X)$ is additively
isomorphic to the Abelian group $(\Z[t]/nt)^n$.

For symplectic orbifolds which are quotients of tori, the equivariant
cohomology version of our product agrees with that defined by Goldin, 
Holm and Knutson in \cite{GHK:07}. The results of this paper may be 
viewed as a method (using 
different techniques) of 
extending their work to non-Abelian group actions.

When $X$ is a point and $G$ is finite, $K_G(I_G(X))$ is additively isomorphic to $K_G(G)$.
The orbifold product then endows an exotic product on $K_G(G)$. This product was previously studied by Lusztig \cite{Lu:87} in the context of Hecke algebras. Furthermore, Lusztig's ring $K_G(G)$ admits an interpretation \cite{AtSe:89}  as the Verlinde algebra of the finite group $G$ at level $0$ (see also \cite{KaPh:07}).

\subsection{Acknowledgments} This paper began during discussions of
the three authors at the BIRS workshop, \emph{Recent Progress on the
moduli space of curves} held March 16--20, 2008. The authors are
grateful to the organizers for the invitation. The first author was
supported by a University of Missouri Research Leave, MSRI and NSA
grant H98230-08-1-0059. 
The first author is grateful to UC Berkeley and MSRI for
hosting him while on leave. The second and third authors were
partially supported by NSF grant DMS-0605155.  The first author is 
grateful to Jonathan Wise and Maciej Zworski for helpful discussions.  
The second author thanks Jeffrey Humphreys for helpful discussions.

\section{Background}
{\bf Conventions:}
In this paper all schemes and algebraic spaces are assumed to be of
finite type over the complex numbers $\C$. All algebraic groups are
assumed to be linear, that is they are isomorphic to closed subgroups
of $\GL_n(\C)$ for some $n$.  We will sometimes use the term {\em linear algebraic group} for emphasis.

However, most of the formalism we develop
also works with equivariant cohomology replacing equivariant Chow
groups, for Lie group actions on almost complex manifolds.

We introduce some notation associated to groups which we will need. 
If $G$  is an algebraic group, we denote the Lie algebra of $G$ by $\Lie(G)$. 
For all $m$ in $G$, let $Z_G(m)$ denote the \emph{centralizer of $m$ in $G$}. For all $\bm = (m_1,\ldots,m_n)$ in $G^n$, denote the \emph{centralizer of $\bm$ in $G$} by $Z_G(\bm)$. 
It consists of all elements commuting with $m_i$ for all $i=1,\ldots,n$. 
For all $n$, the set $G^n$ has a \emph{(diagonal) conjugation action of $G$} defined by $g\cdot (m_1,\ldots,m_n) := (g m_1 g^{-1},\ldots,g m_n g^{-1})$ for all $g$ and $m_i$ in $G$.
A $G$-orbit $\Phi$ of $G^n$ is a called a \emph{diagonal conjugacy class (of length $n$)}, while $\Phi(\bm)$ denotes the diagonal conjugacy class containing $\bm$ in $G^n$. The set of  diagonal conjugacy classes of length n is denoted by $\G{n}$.

\subsection{Group actions and quotient stacks}  
In this paper we will consider three related notions for the action of an 
algebraic group $G$ on a scheme (or more generally algebraic space) $X$.

\begin{defi}
Let $G$ be an algebraic group acting on an algebraic space $X$.  The
\emph{inertia group scheme} $I_G(X)$ is
defined as $$I_G(X):=\{(g,x)| gx =
x\}\subseteq G\times X.$$
\end{defi}
\begin{remark}
If $X$ is an algebraic space then $I_G(X)$ is also an algebraic
space. However the map $I_G(X) \to X$ is representable in the
category of schemes; i.e $I_G(X)$ is an $X$-scheme. For this reason
we refer to $I_G(X)$ as the inertia group scheme even when $X$ 
is algebraic space.
\end{remark}

\begin{defi}  
Let $G$ be an 
algebraic group acting on an algebraic space $X$.
\begin{enumerate}
\item[(i)] We say that $G$ acts \emph{properly} on $X$ if the map $G \times X
\to X \times X$, defined by $(g,x) \mapsto (x,gx)$ is proper.

\item[(ii)] We say that $G$ acts with \emph{finite stabilizer} if the
projection $I_G(X) \to X$ is finite.

\item[(iii)] We say that $G$ acts 
\emph{quasi-freely} 
if the projection $I_G(X) \to X$ is quasi-finite.
\end{enumerate}
\end{defi}

Since $G$ is affine, the map $G \times X \to X \times X$ is finite
if it is proper. The projection $I_G(X) \to X$ is obtained from the
map $G \times X \to X \times X$ by base change along the diagonal
morphism $X \to X \times X$. Hence (i) implies (ii). Moreover the
geometric fibers of the map $I_G(X) \to X$ are the stabilizer groups
and condition (iii) is equivalent to the requirement that the
stabilizer group of any geometric point is finite.  
Since we work in characteristic 0, the quotient stack $[X/G]$ is a
Deligne-Mumford stack (DM stack) if and only if $G$ acts 
quasi-freely.
If $G$ is a finite group, then the action is automatically proper. In general,
$G$ acts properly if and only if $[X/G]$ is a separated DM stack.

In order to construct the orbifold product, we need to push-forward
along the morphism $I_G(X) \times_X I_G(X) \to I_G(X)$. As a result we
require throughout most of the paper that $G$ acts with finite stabilizer.  We also remark that the condition that $G$ act with finite
stabilizer is the necessary separation hypothesis required for the
existence of a coarse moduli space of the quotient stack $[X/G]$
\cite{KeMo:97}.

\begin{defi}
Following \cite{EHKV:01}, we say that a stack $\X$ is a
\emph{quotient stack} if $\X$ is equivalent to a stack of the form
$[X/G]$, where $X$ is an algebraic space and $G$ is a 
{linear} algebraic
group. 
\end{defi}
 Most stacks that naturally arise in
algebraic geometry are quotient stacks. The papers \cite{EHKV:01} and 
\cite{Tot:04}
deal with criteria for determining when a stack is a quotient stack.

\begin{defi}
An \emph{algebraic orbifold} is a smooth DM stack which is
generically represented by a scheme; i.e., the automorphism group at a
general point is trivial.
\end{defi}
\begin{prop}\cite{EHKV:01}
Any algebraic orbifold is a quotient stack.
\end{prop}

\subsection{The inertia group scheme and inertia stack}

\begin{defi}
If $G$ is an algebraic group acting on an algebraic space $X$,
then it induces a $G$-action on $G\times X$ via $$g\cdot (m,x) := (g m g^{-1}, g x).$$
This action preserves $I_G(X)$,
and the quotient $$\IX = [I_G(X)/G]$$ is the \emph{inertia stack
of $\X = [X/G]$}.
\end{defi}
\begin{remark}
If $G$ acts with finite stabilizer on $X$, then the
projection map $\IX \to \X$ is finite.
\end{remark}

\begin{defi}
Let $\Psi$ be a conjugacy class in $G$ and let 
$$I(\Psi) := \{(g,x)|gx =x, g \in \Psi\} \subseteq G\times X.$$ 
\end{defi}
If $G$ acts 
quasi-freely
(in particular
if the action is proper), then $I(\Psi) = \emptyset$ unless $\Psi$
consists of elements of finite order. Since we work in characteristic
$0$, any element of finite order is semi-simple, so its conjugacy class is
closed in $G$ by \cite[Theorem 9.2]{Bor:91}. It follows that $I(\Psi)$
is closed in $I_G(X)$, since it is the inverse image of $\Psi$ under
the projection $I_G(X) \to G$.
\begin{lemma} \label{lem.finite}
If $G$ acts 
quasi-freely,
then all but finitely many of the
$I(\Psi)$ 
are empty.
\end{lemma}
\begin{proof}
Without loss of generality, we may assume that $G$ acts transitively on
the set of connected components of $X$. Let $X^0$ be a connected
component,
and let $U \subset X^0$ be an open set over which the fibers of the
map $I_G(X) \to X$ are finite and flat (and hence \'etale). Let $W =
GU$. Over the $G$-invariant open set $W$, all stabilizers are conjugate
to a fixed finite subgroup $H \subset G$. The complement,  $X \smallsetminus W$
is a union of $G$-invariant subspaces of strictly smaller
dimension. By Noetherian induction, it suffices to prove the
proposition for the open set $W$. Thus we are reduced to proving the
proposition under the assumption that the stabilizer at every point of $X$
is conjugate to a fixed subgroup $H \subset G$.

Let $\Psi \subset G$ be a conjugacy class. Under the 
assumptions on 
$X$, we have $I(\Psi)  = \emptyset$ unless $\Psi \cap  H \neq \emptyset$.
Since $H$ is finite, there can be only finitely many such 
$\Psi$.
\end{proof}
\begin{prop} \cite{EdGr:05} \label{prop.inertiadecomp} 
If $G$ acts
quasi-freely
on $X$, then $I_G(X)$ is the disjoint union
of the finitely many non-empty $I(\Psi)$. 
 In particular, the
$I(\Psi)$ 
 are disjoint sums of connected components of $I_G(X)$.
\end{prop}
\begin{proof} 
Since distinct conjugacy classes are disjoint, the $I(\Psi)$
are
also disjoint. As noted above, the $I(\Psi)$
are also closed. By
definition, every closed point in $I_G(X)$ lies
on some non-empty $I(\Psi)$. Since there are only finitely many
such $I(\Psi)$,
 it follows that the complement of the union of
the $I(\Psi)$ 
is a Zariski open set which contains no closed
points. Since we work over an algebraically closed field, the
complement must be empty.
\end{proof}

\subsection{Multiple inertia schemes}
\begin{defi}
Let $\II_G(X) = I_G(X) \times_X I_G(X)$.  As a set, we have
$$\II_G(X) = \{(g_1,g_2,x)|g_1x = g_2x =x\}.$$ We refer to
$\II_G(X)$ as the 
\emph{double inertia scheme.}  
More generally, we define $$\Imult{l}_G(X):=\underbrace{I_G(X)\times_X \cdots \times_X I_G(X)}_{l \text{ times}} \subseteq G^l\times X.$$
The multiple inertia $\Imult{l}_G(X)$ has $G$-action taking $$m\cdot(g_1,g_2,\dots,g_l,x) := (m g_1 m^{-1}, m g_2 m^{-1},\dots,m g_l m^{-1}, m x).$$
We call the quotient stack
$[\II_G(X)/G]$ the 
\emph{double inertia stack }
$\IIX$. It is
equivalent to the fiber product 
$\IX \times_\X \IX$.  

\end{defi}
\begin{defi}
Given an $l$-tuple of elements $\bg:=(g_1, g_2,\dots,g_l) \in G^l$ define $\Phi(g_1,g_2,\dots,g_l)
\in \G{l}$ to be their orbit under the diagonal action of $G$ by
conjugation on each factor.  Define $$\Imult{l}(\Phi):=\{(g_1,g_2,\dots,g_l,x)|\bg\in \Phi\}\subseteq \Imult{l}_G(X)$$ 
\end{defi}

Clearly, if $G$ acts 
quasi-freely,
then $\Imult{l}(\Phi(\bg))$
is empty unless $g_i$ has finite order for all $i =1, \ldots , l$. A key observation
is that something stronger holds.
\begin{lemma} \label{lem.finitegroup}
If 
$(m_1, \ldots , m_l , x) \in \II_G(X)$, 
then $H = \langle m_1, \ldots , m_l \rangle$ is a
finite group.
\end{lemma}
\begin{proof}
  If $h \in H$, then $hx = x$ (since $m_ix = x$ for all $m_i$). Since $G$ acts
quasi-freely,
we conclude that $H$ must be a finite group.
\end{proof}
\begin{lemma}If $G$ acts 
quasi-freely
on $X$, then 
$\Imult{l}(\Phi)$ is closed in $\Imult{l}_G(X)$ for all $\Phi$ in $\G{l}$.
\end{lemma}
\begin{proof}
If $\Phi \in \G{l}$ is a diagonal conjugacy class, then
$\Itwo(\Phi)$ is the inverse image of $\Phi$ under the projection
$\Imult{l}_G(X) \to     \G{l}$
Thus to prove the proposition it suffices to
show that conjugacy classes of 
$l$-tuples of semi-simple elements are closed
in $G^l$. Since an $l$-tuple $(g_1,\ldots , g_l) \in G^l$ normalizes
the diagonal subgroup $G \subset G^l$, we can again invoke
\cite[Theorem 9.2]{Bor:91} to conclude that if $g_1, \ldots , g_l$ are
semi-simple, then $\Phi(g_1,\ldots , g_l)$ is closed in $G^l$.
\end{proof}
\begin{lemma} \label{lem.finite2}If $G$ acts 
quasi-freely
then $\Imult{l}(\Phi)$ is  empty for all but finitely many 
$\Phi$ in $\G{l}$.
\end{lemma}
 \begin{proof}
The proof is almost identical to the proof of Lemma~\ref{lem.finite}.
Again by Noetherian induction we may reduce to the case when the stabilizer
at every closed point of $X$ is conjugate to a fixed finite subgroup $H
\subset G$.
If $\Phi \subset G \times G$ is a diagonal conjugacy class, then
for all $\Phi$ in $\G{l}$,
$\Imult{l}(\Phi)  = \emptyset$ unless $\Phi \cap  (H^l) \neq \emptyset$.
Since $H^l$ is finite, there can be only finitely many such $\Phi$. 
\end{proof}

The same argument used in the proof of Proposition
\ref{prop.inertiadecomp} yields a decomposition result for $\Imult{l}_G(X)$.
\begin{prop} \label{prop.inertiadecomp2} 
If $G$ acts
quasi-freely
on $X$, then $\Imult{l}_G(X)$ is the disjoint union
of the finitely many non-empty $\Imult{l}(\Phi)$. 
In particular, the
$\Imult{l}(\Phi)$ 
are disjoint sums of connected components of 
$\Imult{l}_G(X)$.
\end{prop}

\subsection{Equivariant Chow groups}
Equivariant Chow groups were defined in \cite{EdGr:98} for actions of
linear algebraic groups on arbitrary algebraic spaces over a field.
They are algebraic analogues of equivariant cohomology groups and the
formalism of this paper also goes through for equivariant
cohomology. If $G$ is an 
algebraic group and $X$ is a $G$-space, then
in this paper we use the notation $A^*_G(X)$ to denote the infinite
direct sum 
$\bigoplus_{i=0}^\infty A^i_G(X)$, where $A^i_G(X)$ is the
``codimension-$i$'' equivariant Chow group. An element of $A^i_G(X)$ is
represented by a codimension-$i$ cycle on a quotient $X \times_G U$,
where $U$ is an open set in a representation on which $G$ acts freely
and such that the complement of $U$ has codimension more than $i$ in
$V$. The space $X \times_G U$ may be viewed as an approximation of the
Borel construction in equivariant cohomology. 

\begin{remark} \label{rem.needspace} Even if $X$ is a scheme,
the quotient $X \times_G U$ may exist only in the category of
algebraic spaces. For this reason, the natural category for
equivariant intersection theory is that of algebraic spaces of finite
type over a field. By their definition, the basic properties of
equivariant Chow groups follow from the corresponding properties of
ordinary Chow groups of algebraic spaces.  As discussed in
\cite[Section 6]{EdGr:98}, the definition of Chow groups of schemes
given in \cite{Ful:84} extends to algebraic spaces, and the results of
\cite[Chapters 1-6]{Ful:84} can be carried over essentially unchanged.
\end{remark}

Since representations may have arbitrarily large dimension, the groups
$A^i_G(X)$ can be non-zero in arbitrarily high degree. If $G$ acts
freely on $X$,
then $A^*_G(X) = A^*(X/G)$, where $X/G$ is
the quotient in the category of algebraic spaces (which always
exists).  If $G$ acts 
quasi-freely, then $A^i_G(X) \otimes
\Q = 0$ for $i > \dim X$, and $A^i_G(X)\otimes \Q = A^i(X/G) \otimes
\Q$ when the quotient exists in the category of algebraic spaces
\cite[Theorem 3]{EdGr:98}. More generally, \cite[Proposition
  19]{EdGr:98} states that $A^*_G(X)$ may be identified with the Chow 
  groups of the quotient stack $[X/G]$. In particular, they depend only
on the underlying quotient stack and not on the group $G$ and space
$X$.

\begin{remark}
In \cite{EdGr:98}, the notation $A^i_G(X)$ was used for the
``codimension-$i$'' operational Chow group, rather than the
codimension-$i$ group of cycles. However, if $X$ is smooth, then these two
groups are identified. In this paper we work exclusively with
smooth spaces so the notational difference is immaterial.
\end{remark}

\begin{remark}
In this paper we will often consider equivariant Chow groups of smooth but disconnected spaces. If $X  = \coprod_{k=1}^m X_k$, then $A^i_G(X) = \oplus_{i=1}^k
A^i_G(X)$ so any ``codimension-$i$'' cycle is a sum of ``codimension-$i$''
cycles on the each connected component $X_k$.
\end{remark}

Equivariant Chow groups enjoy the same formal properties as
ordinary Chow groups. In particular, for $X$ smooth there is an
intersection product which makes $A^*_G(X)$ a graded, commutative
ring. If $f \colon Y \to X$ is a morphism of smooth varieties, then
there is a pullback $f^* \colon A^*_G(X) \to A^*_G(Y)$ which is a ring
homomorphism. If $f$ is proper, then there is a push-forward $f_* \colon
A^*_G(Y) \to A^*_G(X)$ which shifts degrees by the relative
codimension of the morphism $f$. 

If $G$ acts properly on $X$, then there is a pushforward isomorphism
\cite[Theorem 3]{EdGr:98} $p_*\colon A^*_G(X)\otimes \Q \to
A^*(X/G)\otimes \Q$, where $X/G$ is the geometric quotient (which
always exists in the category of algebraic spaces by \cite{KeMo:97}). Let $pr\colon X/G
\to \Spec \C$ be the projection to a point.
\begin{defi} \label{def:quotientdegree} If $G$ acts properly on $X$
and the quotient $X/G$ is complete, then we define the {\em quotient degree}
map $\int_{[X/G]} \colon A^*_G(X)\otimes \Q \to \Q=A^*(\Spec \C)$ by the
formula
$$\int_{[X/G]} \alpha 
:= pr_* p_*\alpha.$$
\end{defi}
\begin{remark}
The stack $[X/G]$ is complete if and only if $G$ acts properly and the quotient $X/G$ is a complete algebraic space.
\end{remark}
\begin{remark} \label{rem.quotientdegree}
The pushforward $A^*_G(X) \otimes \Q \to A_*(X/G)$ commutes with equivariant pushforward for finite
morphisms. This implies two facts about the quotient degree which we will use below:

(i) If $Y \stackrel{f} \to X$ is a finite $G$-equivariant 
morphism
such that $[X/G]$ (and hence $[Y/G]$) 
is complete,
and if 
$\alpha \in A^*_G(Y)\otimes \Q$, 
then
$$\int_{[Y/G]} \alpha = \int_{[X/G]} f_*\alpha.$$

(ii) If $ \sigma \colon X \to X$ is an automorphism that commutes with the $G$-action, then
$$\int_{[X/G]} \sigma^*\alpha = \int_{[X/G]} \alpha$$
for all classes $\alpha \in A^*_G(X)\otimes \Q$.
\end{remark}

Equivariant vector bundles have equivariant Chern classes with values
in the equivariant Chow ring. Equivariant Chern classes have the same
formal properties as ordinary Chern classes.  The only difference is
that, because $A^i_G(X)$ may be non-zero in arbitrarily high degree,
the Chern character and Todd classes must, a priori, be viewed as
elements in the formal completion $\prod_{i =0}^{\infty}
A^i_G(X)\otimes \Q$. However, in this paper we only consider quasi-free 
actions
so the formal completion is the same as
$A^*_G(X) \otimes \Q$.  

We will often consider
$G$-equivariant vector bundles on non-connected algebraic spaces
whose rank varies on the connected components. 
\begin{defi}\label{def:euler}
If $V$ is such a bundle
on a space $X$, then we use the notation 
$\euler(V)$
for the Euler class of $V$, that is the class
in $A^*_G(X)$ whose restriction to equivariant the Chow group of each
connected component is the top Chern class of the restriction of $V$
to that component.
\end{defi}

If $Z \subset G$ is a closed subgroup and $X$ is a $Z$-space, then we
write $G \times_Z X$ for the quotient of the $(G \times Z)$-space $G
\times X$ by the subgroup $1 \times Z$, where $G \times Z$ acts by the
rule $(k,z)\cdot (g,x) = (kgz^{-1}, zx)$. The quotient has an action
of $G$ and we may identify $A^*_Z(X)$ with $A^*_G(G \times_Z X)$
(\cite[Proposition 3.2a]{EdGr:00}).  We refer to this identification as
\emph{Morita equivalence}. If $X$ is also a $G$-space, then $G \times_Z X =
G/Z \times X$ and flat pullback along the morphism $G/Z \times X \to
X$ induces a restriction morphism $A^*_G(X) \to 
A^*_Z(X)$.

\subsection{Equivariant $K$-theory}
Let $G$ be an algebraic group acting on an algebraic space $X$. We use
the notation $K_G(X)$ to denote the Grothendieck ring of
$G$-equivariant vector bundles on $X$.  An element $K_G(X)$ is {\em
  positive} if it is equivalent to a positive integral sum of classes
of equivariant vector bundles. An element is \emph{non-negative} if it
is either 0 or positive. If $\alpha$ is a non-negative, then its Euler
class is well defined, as is the corresponding $K$-theory class
$\lambda_{-1}(\alpha^*)$.

Given a morphism $f \colon Y
\to X$ of $G$-spaces, there is a naturally defined pullback $f^* \colon
K_G(X) \to K_G(Y)$. In order to construct the twisted product on
equivariant $K$-theory we need the existence of pushforwards for
finite local complete intersection morphisms of smooth spaces.  A
sufficient condition for this to hold is if $Y$ and $X$ satisfy the
equivariant resolution property---that is every $G$-coherent sheaf is
the quotient of a $G$-equivariant locally free sheaf \cite[Section
3]{Koc:98}. By Thomason's resolution theorem \cite{Tho:87a}, the
equivariant resolution property holds if $X$ satisfies the
non-equivariant resolution property. The resolution property is known
to hold for smooth schemes, but no general result exists for algebraic
spaces. 

In order to prove associativity of the orbifold product, we
need to use the equivariant self intersection formula for finite local
complete intersection morphisms.  This follows from the excess intersection
formula for $G$-projective morphisms proved by K\"ock
\cite[Theorem 3.8]{Koc:98} for schemes which satisfy the resolution property.
Consequently, when we work in
equivariant $K$-theory we will assume that
we work with smooth schemes rather than smooth algebraic spaces.

Suppose that $G$ acts properly on a scheme $X$ with geometric quotient $X/G$ 
(which need not be a scheme).
Let $\pi \colon X \to X/G$ be the quotient map. By \cite[Lemma 6.2]{EdGr:08}
the assignment ${\mathcal E} \to (\pi_*{\mathcal E})^G$ is an exact functor from the category $G$-equivariant vector bundles on $X$ to the category of coherent sheaves on $X/G$.
\begin{defi} \label{def.quotienteulerchar}
If $G$ acts properly on $X$ and the quotient $X/G$ is complete, then we define the 
{\em  quotient Euler characteristic} map
$\chi_{[X/G]} \colon K_G(X) \to \Z$ by the formula
$$\chi_{[X/G]}({\mathcal E}) = \sum_{i} (-1)^i \dim H^i(X/G, (\pi_*{\mathcal E})^G).$$
\end{defi}
\begin{remark} \label{rem.quotientchi}
As is the case for the quotient degree, the quotient Euler characteristic commutes with finite equivariant morphisms and is invariant under automorphisms which commute with the action of $G$.
\end{remark}

If $Z \subset G$ is a closed subgroup and $X$ is a $G$-space, then we
again have a Morita equivalence identification of $K_Z(X) = K_G(G
\times_Z X)$ as described in \cite[Proposition 3.2(a)]{EdGr:00}.  When $X$ is also a
$G$-space, then $G \times_Z X = G/Z \times X$ and the pullback along
the $G$-equivariant morphism $G/Z \times X \to X$ corresponds to the
restriction map $K_G(X) \to K_Z(X)$.

As explained in \cite[Section 3.2]{EdGr:05}, the Morita equivalence
identification of equivariant $K$-theory follows from an explicit
equivalence between the category of $Z$-locally free sheaves on $X$ and
$G$-locally free sheaves on $G \times_Z X$.  If $V$ is $G$-module on
$G \times_Z X$, then its pullback to $G \times X$ is a $(G\times
Z)$-module on $G \times X$. The subsheaf, ${\mathscr V}$ of $G \times
1$-invariant sections is a $Z$-module on $X$.

\subsection{Fixed loci, conjugacy classes and Morita equivalence}
Let $G$ be an algebraic group acting on an algebraic space $X$.
Consider a diagonal conjugacy class $\Phi$ in $\G{l}$. 
Given ${\bm} = (m_1, \ldots , m_l) \in \Phi$,
let $X^{\bm}$ be the intersection
of the fixed loci 
$X^{m_1} \cap \ldots\cap X^{m_l}$. 
Define a map $G \times X^{\bm} \to \Imult{l}(\Phi)$ by 
$(g,x) \mapsto  (gm_1g^{-1}, \ldots , g m_l g^{-1},gx)$.

\begin{lemma} \label{lem.sphi} (cf. \cite[Lemma 4.3]{EdGr:05})
The map $G \times X^{\bm} \to \Imult{l}(\Phi)$ is a $Z_{\bm}$-torsor,
where
$Z_{\bm}$ acts by $z\cdot(g,x) = (gz^{-1}, zx)$. In particular,
$\Imult{l}(\Phi)$ is smooth if $X$ is smooth.
\end{lemma}

As a consequence we obtain the following decompositions of $A^*(\IX)$
(resp.
$K(\IX)$) and $A^*(\IIX)$ (resp. $K(\IIX)$).
\begin{prop} \label{prop.chowdecomp}
$$A^*(\IX) = \bigoplus_{\Psi} A^*_{Z_G(m)}(X^m) $$
$$K(\IX)  = \bigoplus_{\Psi} K_{Z_G(m)}(X^m),$$
where the sum is over every conjugacy class $\Psi$ of $G$ such 
that $I(\Psi)\neq \emptyset$, and $m$ is a choice of representative for each $\Psi$.

Likewise
$$A^*(\IIX) = \bigoplus_{\Phi}
A^*_{Z_G(m_1,m_2)}(X^{m_1,m_2})$$
$$K(\IIX) = \bigoplus_{\Phi}
K_{Z_G(m_1,m_2)}(X^{m_1,m_2}),$$
where the sum is over all diagonal conjugacy classes $\Phi$ in $\G{2}$ such that $\Itwo(\Phi) \neq \emptyset$, and where $(m_1,m_2)$ is a choice of a representative for each $\Psi$.
\end{prop}

\begin{defi} 
We define \emph{multiplication maps} of the form $\mui \colon G^l \to G^{l-1}$ 
defined by $ (g_1, \ldots , g_l) \mapsto (g_1, \ldots , g_{i-1}, g_i g_{i +1},
g_{i +2}, \ldots g_l)$, where $i\in\{1,\ldots,l-1\}$. 

We also define  \emph{evaluation maps} of the form $\ej \colon G^l \to G^{l-1}$ defined by $ (g_1, \ldots , g_l) \mapsto (g_1, \ldots g_{j-1}, g_{j+1}, \ldots g_l)$, where $j\in\{1,\ldots,l\}$. 

Since these maps commute with the diagonal  action of $G$ on $G^l$ and $G^{l-1}$, they induce maps, also denoted by the same symbols,  $\mui \colon \Imult{l}(\Phi) \to \Imult{l-1}(\mu(\Phi))$ and $\ej \colon \Imult{l}(\Phi) \to \Imult{l-1}(\ej(\Phi))$.
\end{defi}

Now let $V$ be coherent $G$-module on $\Imult{l-1}(\mui(\Phi))$
(resp. $\Imult{l-1}(\ej(\Phi))$). Then $V$ pulls back to a coherent $G$-module 
$W = \mui^*V$ (resp. 
$\ej^*V$
) on $\Imult{l}(\Phi)$. Given 
${\bm} =(m_1, \ldots m_l) \in \Phi$, let ${\mathscr W}$ be the Morita
equivalent 
$Z_{\bm}$-module on $X^{\bm}$. Likewise let ${\mathscr V}$ be the
$Z_{\mui({\bm})}$ (resp. $Z_{\ej({\bm})}$)-module on
$X^{\mui({\bm})}$ (resp. $X^{\ej({\bm})}$) Morita
equivalent to $V$. 
There is an inclusion of centralizers $Z_{\mathbf
  m} \subset Z_{\mui({\bm})}$ (resp. $Z_{\bm} \subset
Z_{\ej({\bm})}$) as subgroups of $G$. Let ${\mathscr V}|_{X^{\mathbf
    m}}$ be the $Z_{\bm}$-module on $X^{\bm}$ obtained by
first restricting the action to a $Z_{\bm}$-action and then
pulling back via the natural inclusion $X^{\bm}
\hookrightarrow X^{\mui({\bm})}$ (resp. $X^{\bm}
\hookrightarrow X^{\ej({\bm})}$).
\begin{lemma} \label{lem.morequivincl}
With the notation as above ${\mathscr W} = {\mathscr V}|_{X^{\bm}}$.
\end{lemma}
\begin{proof}
We only prove the statement for the multiplication map $\mui$,  
as the
proof for the evaluation map is essentially identical.
The proof follows from the fact that the diagram of torsors for the groups
$Z_{\bm}$ and $Z_{\mui({\bm})}$
$$\begin{array}{ccc}
G \times X^{\bm} & \hookrightarrow & G \times
X^{\mui({\bm})}\\
\downarrow & & \downarrow \\
I({\Phi})&  \stackrel{\mui}\to & I(\mu(\Phi))
\end{array}
$$ 
commutes and the upper horizontal arrow is $G \times Z_{\mathbf
  m}$-equivariant.
\end{proof}

Lemma \ref{lem.sphi} also yields the following useful proposition.
\begin{prop}\label{prop:IsomOfInertias}
If $\Phi$ in $\G{\ell+1}$ is the diagonal conjugacy class of $(m_1,\ldots,m_{\ell+1})$, where $\prod_{i=1}^{\ell+1} m_i = 1$, then $\esub{\ell+1}:I(\Phi)\to I(\esub{\ell+1}(\Phi))$ is a $G$-equivariant isomorphism which induces a $G$-equivariant embedding $\Imult{\ell}_G(X)\to \Imult{\ell+1}_G(X)$.
\end{prop} 

\section{Inertia group scheme products} \label{sec.inertialproducts}

{\bf Background hypotheses.} Throughout this section, we assume that
$X$ is a smooth algebraic space when working with equivariant Chow
groups and a smooth scheme when working with equivariant $K$-theory.
If a group $G$ acts on a scheme or space $X$, then we assume that the
action has \emph{finite stabilizer}.\\

If we view $I_G(X)$ as a group scheme over $X$, there are 
three maps
$$\II_G(X) = I_G(X) \times_X I_G(X) \to I_G(X).$$ Let $e_1,e_2$
be the projections onto the first and second factors respectively and
let
$\mu$ be the multiplication map. 
Under the assumption that $G$ acts on $X$ with
finite stabilizer, all three of the above maps are finite.
\begin{defi}
 Given a
class 
$c \in A_G^*(\II_G(X))$, 
we may define a binary operation $\star_c$ on 
$A^*_G(I_G(X))$ by the formula
\begin{equation} \label{eq.cprod}
\alpha \star_c \beta = \mu_*(e_1^*\alpha \cdot e_2^*\beta
\cdot c).
\end{equation}
If we let $\X = [X/G]$, then identifying $A^*_G(I_G(X)) = A^*(\IX)$, the $\star_c$ product may be viewed
as a product on $A^*(\IX)$.
\end{defi}

When $X$ has the resolution property for coherent sheaves (e.g., if $X$
is a scheme), then, given a class 
$\kclass \in K_G(\II_G(X))$
we may define a product $\star_\kclass$ on $K_G(I_G(X))$ by the
formula
\begin{equation} \label{eq.eprod}
{\mathscr F} \star_\kclass {\mathscr G} = \mu_*(e_1^*{\mathscr F}
\otimes e_2^*{\mathscr G} \otimes \kclass)
\end{equation}
Again this product corresponds to a product on
$K(\IX)$.

\begin{defi} \label{def.metric}
Suppose that $[X/G]$ is complete (i.e., $G$ acts properly on $X$ and $X/G$ is complete).
 Let $\sigma:I_G(X) \to I_G(X)$ be the involution taking $(g,x)$ to $(g^{-1},x)$.  We define the 
{\em pairing}
$\eta \colon 
A^*_G(I_G(X))\otimes\Q \bigotimes A^*_G(I_G(X))\otimes\Q 
 \to \Q$ as 
 $$\eta(\alpha_1,\alpha_2) := \int_{[I_G(X)/G]} \alpha_1 \cdot \sigma^*\alpha_2$$
 for any classes $\alpha_1, \alpha_2 \in  
A^*_G(I_G(X))\otimes\Q
$.  
Similarly we define the 
{\em pairing}
$\eta \colon K_G(X) \otimes K_G(X) \to \Z$
as 
$$ \eta({\mathscr F}, {\mathscr G}) = \chi_{[I_G(X)/G]}({\mathscr F} \otimes \sigma^*{\mathscr G})$$
\end{defi}

\begin{remark}
Since $\sigma$ is an involution, observe that $\alpha_2 \cdot \sigma^*\alpha_1 = \sigma^*(\alpha_1 \cdot \sigma^*\alpha_2)$.
Hence by Remark \ref{rem.quotientdegree} 
$\eta(\alpha_1,\alpha_2) = \eta(\alpha_2,\alpha_1)$.
\end{remark}

\begin{defi} \label{def.frobproperty}
If $[X/G]$ is complete,
the $\star_c$ product 
on $A_G^*(I_G(X))\otimes\Q$ 
is {\em Frobenius} if 
$$\eta(\alpha_1 \star_c \alpha_2, \alpha_3) = \eta(\alpha_1, \alpha_2 \star_c \alpha_3)$$
for all classes $\alpha_1, \alpha_2, \alpha_3 \in 
A^*_G(I_G(X))\otimes \Q
$.
We define the Frobenius property of the $\star_\kclass$ product on $K_G(X)$ analogously.
\end{defi}
\begin{prop}\label{prop:frob}
A sufficient condition for the $\star_c$ product to be Frobenius is 
that the class $c\in A^*(\II_G(X))$ satisfy \emph{cyclic invariance}; that is, 
$$\tau^*(c)  = c,$$
where $\tau:\II_G(X) \to \II_G(X)$ is the map taking $(g_1,g_2,x)$ to $(g_2,(g_1g_2)^{-1}, x)$. 

The analogous statement holds for the $\star_{\mathscr E}$ product.
\end{prop}
\begin{proof} 
This follows from the projection formula, the automorphism invariance of the quotient degree (Remark \ref{rem.quotientdegree}) and the following simple properties of $\tau$:
\begin{equation*}
\begin{array}{rlrl}
e_2 &= e_1\circ \tau & \qquad e_1 &= \sigma\circ \mu \circ \tau \\
 \qquad\sigma \circ \mu &= e_2\circ \tau.
\end{array}
\end{equation*}
We have
\begin{align*}
\eta(\alpha_1\star_c \alpha_2, \alpha_3) & =  \int_{[I_G(X)/G]} \mu_*(e_1^*\alpha_1 \cdot e_2^* \alpha_2 \cdot c)\cdot \sigma^*\alpha_3\\
& = \int_{[I_G(X)/G]} \mu_*(e_1^*\alpha_1 \cdot e_2^* \alpha_2 \cdot c\cdot \mu^*\sigma^* \alpha_3)\\
& = \int_{[\II_G(X)/G]}(\tau^*\mu^*\sigma^*\alpha_1 \cdot \tau^*e_1^* \alpha_2 \cdot \tau^*c \cdot \tau^*e_2^*\alpha_3)\\
&= \int_{[\II_G(X)/G]}(e_1^* \alpha_2 \cdot e_2^*\alpha_3\cdot c \cdot \mu^*\sigma^*\alpha_1)\\
&=  \int_{[I_G(X)/G]} \mu_*(e_1^* \alpha_2 \cdot e_2^*\alpha_3\cdot c) \cdot\sigma^*\alpha_1 
\\
&=\eta(\alpha_2 \star_c \alpha_3,\alpha_1) =\eta(\alpha_1, \alpha_2 \star_c \alpha_3) 
\end{align*}
\end{proof}
The decomposition of $I_G(X)$ into pieces $I(\Psi)$ and the decomposition of
$\II_G(X)$ into pieces $\Itwo(\Phi)$ gives a more refined description of
the $\star_c$ and $\star_\kclass$ products. We will use this
description to give sufficient conditions for the products to be
commutative and associative.
\begin{defi}
Given a conjugacy class $\Psi \in \G{1}$ 
and a class $\alpha \in
A_G^*(I_X)$, let $\alpha_\Psi$ be 
the
component in the summand
$A^*_G(I(\Psi))$ of $A^*_G(I_X)$. 

Likewise if $\Phi\in \G{2}$ is a diagonal conjugacy class,
we let $c_\Phi \in A^*_G({\mathbb
  I}_G(X))$ denote the component of a class $c$ in the summand
$A^*_G(\Itwo(\Phi))$.
We will use similar notation for elements of the equivariant
Grothendieck group.
\end{defi}

The $\star_c$ product can be expressed as a sum over diagonal
conjugacy classes in $G \times G$ 
as follows.   Let 
$\Phi\in\G{2}$
be a diagonal conjugacy class
and let $\Psi_1 = e_1(\Phi)$ and $\Psi_2 = e_2(\Phi)$.
Likewise, let $\Psi_3 = \mu(\Phi)$.
Given classes
$\alpha_1 \in A^*_G(I(\Psi_1))$ and $\alpha_2 \in A^*_G(I(\Psi_2))$
the product $\alpha_1 \star_c \alpha_2$ will have a contribution in
$A^*_G(I(\Psi_3))$ given by 
\begin{equation} \label{eq.phiprod}
\mu_*(e_1^*\alpha_1 \cdot e_2^*\alpha_2 \cdot c_{\Phi})
\end{equation}
The product $\alpha_1 \star_c \alpha_2$ is obtain by summing over
terms of the form \eqref{eq.phiprod} for all diagonal conjugacy
classes
$\Phi$ such that $e_1(\Phi) =\Psi_1$ and $e_2(\Phi) =\Psi_2$.

Not all choices of a class $c \in A^*_G(\II_G(X))$ produce
interesting products. We begin with a necessary and sufficient
condition for $\star_c$ to have an identity.
\begin{prop} \label{prop.chowidentity}
Let ${\mathbf 1} \in A^*_G(I_G(X))$ be the fundamental class of $I(\{1\})= X
\subset I_G(X)$. Then ${\mathbf 1}$ is the identity for the $\star_c$ product
if and only if $c_\Phi = [\Itwo(\Phi)]$ for all diagonal conjugacy classes
such that $e_1(\Phi) = \{1_G\}$ or $e_2(\Phi) = \{1_G\}$.
\end{prop}
\begin{proof}
If $\beta \in A^*_G(I(\Psi))$, then the only diagonal conjugacy class
$\Phi \in \G{2}$ 
satisfying $e_1(\Phi) = \{1\}$ and $e_2(\Phi) = \Psi$ is the
class $\Phi = (1, \Psi) \subseteq G \times G$. Thus
the only contribution to ${\mathbf 1} \star_c \beta$ is in the $\Psi$
component. On $\Itwo(\Phi)$ $\mu= e_2$ and $e_1^*{\mathbf 1} = [\Itwo(\Phi)]$.
Thus the projection formula gives
$${\mathbf 1} \star_c \alpha = \alpha \cdot \mu_* c_{\Phi}.$$
Since 
$\mu = e_2$
is an isomorphism $\Itwo(\Phi)\to I(\Psi)$, it follows
that
$\mathbf{1} \star_c \alpha = \alpha$ if and only if $c_{\Phi} = [\Itwo(\Phi)]$. 
Exchanging the roles of 
$\mathbf{1}$
and $\alpha$ yields the
condition that $c_{\Phi} = [\Itwo(\Phi)]$ for $\Phi = (\Psi,1)$.
\end{proof}
An essentially identical argument yields the following criterion for
the $\star_{\kclass}$ product. 
\begin{prop} \label{prop.kidentity}
The class $[{\mathscr O}_X] \in K_G(I_G(X))$ is the identity for the
$\star_\kclass$ product if and only if $\kclass_\Phi = [{\mathscr
  O}_{\Itwo(\Phi)}]$
for all diagonal conjugacy classes $\Phi$ 
in $\G{2}$
such that $e_1(\Phi) =
\{1_G\}$
or $e_2(\Phi) = \{1_G\}$.
\end{prop}

Next we give a condition that ensures that the $\star_c$ product is
commutative.
Let $i \colon \II_G(X) \to \II_G(X)$ be the involution induced
by the involution on $G \times G$ that exchanges the factors.
\begin{prop} \label{prop.commutative}
A sufficient condition for the $\star_c$ product to be commutative is
that for each diagonal conjugacy class $\Phi$ we have
$$i^*c_{\Phi} = c_{i(\Phi)}.$$ 
Similarly, the $\star_\kclass$ product
is commutative provided that $i^*\kclass_{\Phi} = \kclass_{i(\Phi)}$.
\end{prop}
\begin{proof}
The proof is a straightforward application of \eqref{eq.phiprod} and
the fact that 
$\mu(\Phi)= \mu(i(\Phi))$, 
because if $m_1,m_2 \in G$, then
$m_1 m_2$ and $m_2 m_1$ are conjugate.
\end{proof}
\subsection{A criterion for associativity of the $\star_c$ and
  $\star_\kclass$
products} \label{sec.assoccriterion}
In this section, we give a sufficient condition for the $\star_c$ product to be
associative. An analogous condition, which we do not write down, also
holds for the $\star_\kclass$ product.
We closely follow Section 5 of \cite{JKK:07}.
Given $m_1, m_2, m_3 \in G$, with conjugacy
classes $\Psi_1, \Psi_2, \Psi_3$, respectively, let $\Psi_{12}$ and
$\Psi_{23}$ be the conjugacy class of the products $m_1m_2$ and $m_2
m_3$ and  let $\Psi_{123}$ be the conjugacy class of the product
$m_1m_2m_3$.  Let $\Phi_{1,2}$ and $\Phi_{2,3}$ be diagonal conjugacy
classes of the pairs $(m_1,m_2)$ and $(m_2,m_3)$
respectively. Let $\Phi_{12,3}$ be the diagonal conjugacy
class of the pair $(m_1m_2, m_3)$ and $\Phi_{1,23}$ the diagonal
conjugacy class of the pair $(m_1,m_2m_3)$.
\begin{lemma} \label{lem.assoc}
A sufficient condition for associativity to be satisfied is if 
\begin{align} \label{eq.assoc1}
\mu_*\left(e_1^*\mu_*(e_1^*\alpha_1  \cdot
     e_2^*\alpha_2 \cdot c_{\Phi_{1,2}}) \right. & \left.\cdot e_2^*\alpha_3 \cdot
      c_{\Phi_{12,3}}\right)=\\
 &\qquad \mu_*\left(e_1^*\alpha_1 \cdot e_2^*\mu_*(e_1^*\alpha_2 \cdot
    e_2^*\alpha_3 \cdot c_{\Phi_{2,3}}) \cdot c_{\Phi_{1,23}}\right) \notag
\end{align}
for all diagonal conjugacy classes $\Phi_{1,2}$, $\Phi_{12,3}$, $\Phi_{2,3}$,
$\Phi_{1,23}$ in $\G{2}$
determined by a triple of elements $(m_1,m_2,m_3) \in G^3$ and all classes
$\alpha_1,\alpha_2,\alpha_3 \in A^*_G(I(\Psi_i))$. 
The analogous statement
holds for the $\star_\kclass$ product.
\end{lemma}
\begin{proof}
We only give the proof for the $\star_c$ product, as the proof for the
$\star_\kclass$ product is essentially identical.
Let $\Psi_{1}, \Psi_2,\Psi_3, \Psi_{123}$ be conjugacy
classes in $G$.
Given classes $\alpha_1 \in A^*_G(I(\Psi_1)), \alpha_2 \in
A^*_G(I(\Psi_2)), \alpha_3 \in A^*_G(I(\Psi_3))$, the component of
the product 
$$(\alpha_1 \star_c \alpha_2) \star_c \alpha_3$$ in
$A^*_G(I(\Psi_{123}))$ is the sum
  \begin{equation} \label{eq.12to3} \sum_{\Phi_{1,2}}
    \sum_{\Phi_{12,3}} \mu_*\left(e_1^*\mu_*(e_1^*\alpha_1 \cdot
      e_2^*\alpha_2 \cdot c_{\Phi_{1,2}}) \cdot e_3^*\alpha_3 \cdot
      c_{\Phi_{12,3}}\right),
\end{equation}
where the sum is over all pairs of diagonal conjugacy classes 
$\Phi_{1,2}$, $\Phi_{12,3} $ in $\G{2}$ 
satisfying 
\begin{equation} \label{eq.12to3conj}
\begin{array}{lll}
e_1(\Phi_{1,2})= \Psi_1, \quad & e_2(\Phi_{1,2}) =
\Psi_2, \quad & \mu(\Phi_{1,2}) = e_1(\Phi_{12,3}),\\
 e_2(\Phi_{12,3}) =
\Psi_3, \quad &  \mu(\Phi_{12,3}) = \Psi_{123}.
\end{array}
\end{equation}
Similarly, 
the $\Psi_{123}$ component of the product $\alpha_{1} \star_c
(\alpha_2 \star_c \alpha_3)$ is calculated by the following sum.
\begin{equation} \label{eq.1to23} \sum_{\Phi_{1,23}} \sum_{\Phi_{2,3}}
 \mu_*\left(e_1^*\alpha_1 \cdot e_2^*\mu_*(e_1^*\alpha_2 \cdot
    e_2^*\alpha_3 \cdot c_{\Phi_{2,3}}) \cdot c_{\Phi_{1,23}}\right),
\end{equation}
where the sum is over all pairs of diagonal conjugacy classes 
$\Phi_{1,23}, \Phi_{2,3}$ in $\G{2}$ 
satisfying 
\begin{equation}\label{eq.1to23conj}
\begin{array}{lll} 
e_1(\Phi_{2,3}) =  \Psi_2, \quad & e_2(\Phi_{2,3}) =  \Psi_3,   
\quad &e_1(\Phi_{1,23}) = \Psi_1\\
e_2(\Phi_{1,23}) = \mu(\Phi_{2,3}),  \quad& \mu(\Phi_{1,23}) = \Psi_{123}
\end{array}
\end{equation}
A sufficient condition for the $\star_c$ products $(\alpha_1 \star_c
\alpha_2) \star_c \alpha_3$ and $(\alpha_{1} \star_c \alpha_2)\star_c
\alpha_3$ to be equal is that the terms in sums \eqref{eq.1to23}
and \eqref{eq.12to3} be identified. The assignments
$(m_1,m_2,m_3)
\mapsto \Phi_{1,2}, \Phi_{12,3}$ and $(m_1,m_2,m_3) \mapsto
\Phi_{2,3}, \Phi_{1,23}$ gives a bijection between the set of
of pairs of conjugacy classes satisfying \eqref{eq.12to3conj} and
those satisfying \eqref{eq.1to23conj}.
\end{proof}
Following \cite{JKK:07} the condition of Lemma~\ref{lem.assoc} can be
expressed in terms of excess normal bundles.
Given $m_1, m_2, m_3 \in G$,
let $\Phi_{1,2,3}$ in $\G{3}$ denote the diagonal conjugacy class of the tuple $(m_1, m_2,m_3)$, $\Phi_{12,3}$ in $\G{2}$ the diagonal conjugacy class of 
$(m_1m_2,m_3)$, $\Phi_{1,23}$ in $\G{2}$ the diagonal conjugacy class of $(m_1,m_2m_3)$, $\Phi_{i,j}$ in $\G{2}$ the diagonal conjugacy class of $(m_i,m_j)$ for
$1\leq i < j \leq 3$ and $\Psi_{123}$ the conjugacy class of
$m_1m_2m_3$.
For 
$1 \leq i < j  \leq 3$ 
there are evaluation maps
$e_{i,j} \colon \Imult{3}(\Phi_{1,2,3}) \to
\Itwo(\Phi_{i,j})$. Also, for $1\leq i \leq 3$ there are 
evaluation maps 
$\epsilon_i \colon \Imult{3}(\Phi_{1,2,3}) \to I(\Psi_i)$.
Likewise there are product maps
$\mu_{12,3} \colon \Imult{3}(\Phi_{1,2,3}) \to \Itwo(\Phi_{12,3})$ and
$\mu_{1,23}\colon \Imult{3}(\Phi_{1,2,3}) \to \Itwo(\Phi_{1,23})$ and $\mu_{123} \colon
  \Imult{3}(\Phi_{1,2,3})\to I(\Psi_{123})$.
\begin{lemma}
The diagrams 
\begin{equation} \label{diag.e12} 
\begin{array}{ccc}
\Imult{3}(\Phi_{1,2,3}) & \stackrel{e_{1,2}} \to & \Itwo(\Phi_{1,2})\\
\mu_{12,3}\downarrow & & \mu \downarrow\\
\Itwo(\Phi_{12,3}) & \stackrel{e_1}\to & I({\Psi_{12}})
\end{array}
\end{equation}
and 
\begin{equation} \label{diag.e23}
\begin{array}{ccc}
\Imult{3}(\Phi_{1,2,3}) & \stackrel{e_{2,3}} \to & \Itwo(\Phi_{2,3})\\
\mu_{1,23}\downarrow & & \mu \downarrow\\
\Itwo(\Phi_{1,23}) & \stackrel{e_2}\to & I(\Psi_{23})
\end{array}
\end{equation}
are Cartesian and the horizontal arrows are finite local complete
intersection morphisms.
\end{lemma}
\begin{proof}
Since $X$ is smooth, the maps are l.c.i.~because 
the $\Imult{l}(\Phi)$
and $I(\Psi)$
are all smooth. 
An easy calculation with fixed loci shows that the diagrams are Cartesian.
The map $e_1$ is the composition 
$$\Itwo(\Phi_{12,3})\hookrightarrow  I(\Psi_{12})\times_X I(\Psi_{12}) \to I(\Psi_{12}).$$
By Proposition~\ref{prop.inertiadecomp2}, the first map is an open and closed
embedding. 
We assume that the action
has finite stabilizer so the second map is finite, since it is obtained
by base change from the projection $I_G(X) \to X$. Hence by base change
the map $e_{1,2}$ is finite. An identical argument shows that $e_2$ and
$e_{2,3}$ are also finite.
\end{proof}
Let $E_{1,2}$ be the excess normal bundle for diagram \eqref{diag.e12}
and let $E_{2,3}$ be the excess normal bundle for diagram
\eqref{diag.e23}.
\begin{prop} \label{prop.assoc}
A sufficient condition for the $\star_c$ product to be associative 
is that the following identity holds in $A^*_G(\Imult{3}(\Phi_{1,2,3}))$:
\begin{equation} \label{eq.assoc2}
e_{1,2}^*c_{\Phi_{1,2}} \cdot \mu_{12,3}^*c_{\Phi_{12,3}} \cdot 
\euler(E_{1,2}) 
=
e_{2,3}^*c_{\Phi_{2,3}} \cdot \mu_{1,23}^*c_{\Phi_{1,23}} \cdot 
\euler(E_{2,3})
\end{equation}
for all triples of elements 
$(m_1,m_2,m_3) \in G^3$, and where $\euler$ denotes the Euler class, as in Definition~\ref{def:euler}.
A sufficient condition for the $\star_\kclass$ product to be
associative is that
\begin{equation} \label{eq.kassoc2} 
e_{1,2}^*\kclass_{\Phi_{1,2}} \otimes
  \mu_{12,3}^*\kclass_{\Phi_{12,3}} \otimes \lambda_{-1}(E^*_{1,2}) =
  e_{2,3}^*c_{2,3} \otimes \kclass_{\Phi_{1,23}}^*c_{1,23} \otimes
  \lambda_{-1}(E^*_{2,3})
\end{equation}
\end{prop}
\begin{proof}
Again we only give the proof in equivariant Chow theory as the proof
in equivariant $K$-theory is essentially identical.  We wish to
compare the two sides of \eqref{eq.assoc1}. As in \cite{JKK:07} we
will use the excess intersection formula. However the morphisms we
consider are finite local complete intersection morphisms.  In
equivariant intersection theory, the excess intersection formula follows from
the definition of equivariant Chow groups and the corresponding
non-equivariant excess intersection formula for algebraic spaces
\cite[Theorem 6.5]{Ful:84} (see Remark~\ref{rem.needspace} above).
In equivariant $K$-theory, the excess intersection
formula for finite l.c.i.~morphisms of schemes satisfying the resolution
property follows from \cite[Theorem 3.8]{Koc:98}.

By the
equivariant excess intersection formula for l.c.i.~morphisms, 
if $x \in
A_*^G(\Itwo(\Phi_{1,2}))$, then 
$e_1^*\mu_* x = \mu_{12,3*}(
\euler(E_{1,2}) 
\cdot e_{1,2}^*x)$. Thus the left-hand side of
\eqref{eq.assoc1} can be rewritten as 
\begin{equation} \label{eq.step1}
\mu_*\left[\mu_{12,3*}\left(e_{1,2}^*e_1^*\alpha_1 \cdot e_{1,2}^*e_2^*\alpha_2
  \cdot e_{1,2}^*c_{\Phi_{1,2}} \cdot 
\euler(E_{1,2})
\right) \cdot
e_2^*\alpha_3 \cdot c_{\Phi_{12,3}}\right]
\end{equation}
Since $\mu \circ \mu_{12,3} = \mu_{1,2,3}$ and $e_1 \circ e_{1,2}
=\epsilon_1$,
$e_2 \circ e_{1,2} = \epsilon_2$, we may apply the projection formula
for the map $\mu_{12,3*}$ to rewrite \eqref{eq.step1} as
\begin{equation} \label{eq.step3}
\mu_{1,2,3*}\left(\epsilon_1^*\alpha_1 \cdot \epsilon_2^*\alpha_2
\cdot e_{1,2}^*c_{\Phi_{1,2}}
\cdot 
\euler(E_{1,2}) 
\right)
\cdot \mu_{12,3}^*e_{2}^*\alpha_3 \cdot
\mu_{12,3}^*c_{\Phi_{12,3}}.
\end{equation}
Finally, we note that $\mu_{12,3} \circ e_{2} = \epsilon_3$, so
\eqref{eq.step3} can be rewritten as
\begin{equation} \label{eq.lhs}
\mu_{1,2,3*}\left( \epsilon_1^*\alpha_1 \cdot \epsilon_2^* \alpha_2
  \cdot \epsilon_3^* \alpha_3 \cdot e_{1,2}^*c_{\Phi_{1,2}} \cdot \mu_{12,3}^*
  c_{\Phi_{12,3}} \cdot 
\euler(E_{1,2})
\right).
\end{equation}
A similar calculation shows that the right-hand side of
\eqref{eq.assoc1} can be rewritten  as
\begin{equation} \label{eq.rhs}
\mu_{1,2,3*}\left( \epsilon_1^*\alpha_1 \cdot \epsilon_2^* \alpha_2
 \cdot  \epsilon_3^* \alpha_3 \cdot e_{2,3}^*c_{\Phi_{2,3}} \cdot \mu_{1,23}^*
  c_{\Phi_{1,23}} \cdot 
\euler(E_{2,3})
\right)
\end{equation}
Clearly \eqref{eq.assoc2} implies that \eqref{eq.lhs} and
\eqref{eq.rhs} are equal.
\end{proof}

\section{Logarithmic traces} \label{sec.logtrace} 
Let $X$ be an
algebraic space with the action of algebraic group $Z$. Suppose that
we are given a rank-$n$ $Z$-equivariant vector bundle $V$ on $X$ and
matrices $g_1, \ldots , g_l \in U(n)$ satisfying $\prod_{i = 1}^l g_i
= 1$ which act on the fibers of $V \to X$ and whose action commutes
with the action of $Z$. In this section we define the \emph{logarithmic trace} 
$L(g_i)(V) \in K_Z(X)$ 
and show that 
$$\sum_{i=1}^{l}L(g_i)(V) - V + 
\bigcap_{i=1}^l V^{g_i}$$ 
is represented by a non-negative integral element of 
$K_Z(X)$.

\subsection{The logarithmic trace on a complex vector space} \label{sec.logtracemodule}
\begin{defi}
Let $V$ be an $n$-dimensional complex vector space and let $g \in
\GL(V)$ be an element which lies in some compact subgroup $K \subset
\GL(V)$. This is equivalent to assuming that $g$ is conjugate to a
unitary matrix.
Write the eigenvalues of $V$ as 
$$\exp(2\pi \sqrt{-1} \alpha_1), \ldots
, \exp(2\pi \sqrt{-1} \alpha_n)$$ with 
$0\leq \alpha_i < 1$ for all $i=1,\ldots,n$.
Define the \emph{logarithmic trace} of $g$ by the formula
\begin{equation} \label{eq.logtraced}
L(g) = \sum_{i = 1}^n \alpha_i
\end{equation}
\end{defi}

The key fact we need about the logarithmic trace is the following proposition:
\begin{prop} \label{prop.logtracepositive} \cite{FaWe:06} Let
$V$ be an $n$-dimensional complex vector space and 
let $g_1, \ldots g_l \in \GL(V)$ be elements which lie in a common
compact subgroup $K \subset \GL(V)$ and satisfy $g_1 \ldots g_l =1$.
Then $\sum_{i=1}^l L(g_i)$ is a non-negative integer satisfying the
inequality 
\begin{equation} \label{eq.traceineq}
\sum_{i =1}^l L(g_i) \geq n - n_0
\end{equation}
where $n_0 = \dim \bigcap_{i =1}^l V^{g_i}$ is the dimension of the
invariant subspace.
\end{prop}
\begin{proof}
The fact that $\sum_{i = 1}^l L(g_i) \in  {\mathbb Z}$ follows from the fact
that 
$$\exp(2 \pi \sqrt{-1} L(g_i)) = \det g_i$$ and $\prod_{i =1}^l \det
g_i =1$. Since $K$ lies in a subgroup of $\GL(V)$ isomorphic to $U(n)$,
the inequality \eqref{eq.traceineq} follows from Theorem 2.2 of
\cite{FaWe:06}, 
which states that for any $l$ unitary matrices $A_1,\dots,A_l\in U(n)$ with product equal to the identity matrix $\mathbf{1}$, the sum of the logarithmic traces $\sum_{i=1}^l L(A_i)$ is greater than or equal to $n-n_0$. 
\end{proof}
\subsection{The logarithmic trace on vector bundles}
\label{sec.logtracevectorbundles} Let $X$ be an algebraic space with
the action of an algebraic group $Z$.  The definition of the
logarithmic trace \eqref{eq.logtraced} and the inequality 
\eqref{eq.traceineq} of 
Proposition~\ref{prop.logtracepositive} can be easily generalized to
$Z$-equivariant vector bundles. 
The $K$-theory version of the logarithmic trace
will be used to define the twisted pullback bundles used in the
construction of the inertial group scheme product.
\begin{defi}
Let $V$ be a rank-$n$ vector bundle on $X$ and let $g$ be a 
a unitary automorphism of the fibers of $V \to X$. If we assume  that the action of $g$ 
commutes with action of $Z$ on $V$, 
the eigenbundles for the action of $g$ are all $Z$-subbundles, 
and we define the \emph{logarithmic trace  of $V$} by the formula 
\begin{equation} \label{eq.logtracebund}
L(g)(V) = \sum_{k =1}^{m} \alpha_k V_k \in
K_Z(X) \otimes {\mathbb R}
\end{equation}
on each connected component of $X$.
Here
$\exp(2\pi \sqrt{-1} \alpha_1), \ldots ,\exp(2\pi \sqrt{-1}
\alpha_m)$ are the distinct eigenvalues of $g$ acting on $V$, 
with $0\le \alpha_k <1$ for all $k\in\{1\dots,m\}$,
and $V_1, \ldots , V_m$ are the corresponding eigenbundles.  
\end{defi}
\begin{defi}\label{def:age}
Let $V$ be a rank-$n$ vector bundle on $X$ and let $g$ be a
unitary automorphism of the fibers of $V \to X$ which commutes with
$Z$.  The \emph{age $a_V(g)$ of $g$ on $V$} is the rank of
$L(g)(V)$, which is a locally constant function on 
$X$
\[
a_V(g) := \sum_{k =1}^{m} \alpha_k \; \mathrm{dim}(V_k)
\]
\end{defi} 

\begin{remark}
If $g$ has finite order and acts on a variety $X$, then $a_V(g)$ 
is a locally constant function on $X^g$ which takes values in $\Q$.
\end{remark}
The following proposition is an easy generalization of Proposition 
\ref{prop.logtracepositive}.
\begin{prop} \label{prop.logtracekt} Let ${\bg} = (g_1, \ldots
g_l)$ be an $l$-tuple of elements which lie in a compact subgroup of
a reductive group $H$ and satisfy $g_1 \ldots g_l = 1$.  Let $V$ be
a $Z \times H$-equivariant vector bundle on $H$, where $H$ is assumed
to act trivially on $X$.  Then
$$\sum_{i =1}^l L(g_i)(V) - V + V^{\bg}$$ is a non-negative
element of $K_Z(X)$, where $V^{\bg}$ is subbundle of $V$ invariant
under the action of the $g_i$
for all $i\in\{1,\ldots,l\}$.
In particular, the map
$V \mapsto \sum_{i=1}^l L(g_i)(V) - V + V^{\bg}$ is an additive homomorphism 
$K_{Z \times H}(X) \to K_Z(X)$ 
which takes non-negative elements to non-negative elements.
\end{prop}
To prove the proposition, we first need the following lemma.

\begin{lemma}(cf. \cite[Prop 2.2]{Seg:68b}) \label{lem.vbdecomp}
Let $Z$ be an algebraic group acting 
on
an algebraic space $X$ and 
let $H$ be a linearly reductive group acting trivially on $X$. 
Then any $Z \times H$-equivariant vector bundle $V \to X$ has a canonical 
decomposition 
as a direct sum 
$\bigoplus_{E} {V_E} \otimes E$, where $E$ runs over the set of irreducible $H$-modules and $V_E$ is a $Z$-bundle.
In particular, $K_{Z \times H}(X)$
is canonically isomorphic to $K_Z(X) \otimes \Rep(H)$
\end{lemma}
\begin{proof}[Proof of Proposition~\ref{lem.vbdecomp}]
The proof is more or less identical to the proof given in \cite{Seg:68b}.
If $E$ is an $H$-module, then we can consider the $Z \times H$-module
$ {\mathscr O}_X \otimes E$. 
If $V$ is a $Z \times H$-module, then 
${\mathscr Hom}({\mathscr O}_X \otimes E, V)$ has a natural structure
as $Z \times H$-module. Let $V_E$ be the invariant subbundle
$${\mathscr Hom}({\mathscr O}_X \otimes E,V)^{1 \times H}.$$
There is a natural
map of $Z\times H$-vector bundles $V_E \otimes E \to V$.
As noted in the 
proof
of Proposition 2.2 of \cite{Seg:68b}
the map 
$\bigoplus_E E \otimes V_E \to V$ is an isomorphism, where the sum runs over all irreducible
$H$-modules $E$. 
\end{proof}
\begin{proof}[Proof of Proposition~\ref{prop.logtracekt}]
By Lemma~\ref{lem.vbdecomp} we may assume that $V = V_E \otimes E$, where
$V_E$ is a $Z$-module and $E$ is an irreducible representation of $H$.
Thus $L(g)(V) =  (L(g)(E))(V_E)$.  By
Proposition~\ref{prop.logtracepositive}, $\sum_{i = 1}^l L(g)(E) \geq \dim E -
\dim E^g$, and the proposition follows.
\end{proof}

\begin{remark} \label{rem.canonical}
A priori the class $\sum_{i = 1}^l L(g_i)(V)$ is only well defined up
to torsion. However, the proof of Proposition~\ref{prop.logtracekt}
shows that the there is a canonical choice of integral representative.
More precisely, if $V = \sum_E V_E \otimes E$, then $\sum_{i=1}^l
L(g_i)(V)$ is represented by the integral class $\sum_E (\sum_{i=1}^l
L(g_i)(E))[V_E]$
\end{remark}

\section{Logarithmic restriction and twisted pullback bundles}
\label{sec.logres}

\textbf{Background hypotheses.}
Throughout this section, we will assume that $G$ is an
algebraic
group acting on a \emph{smooth} algebraic space $X$.  There is no 
restriction on the action of $G$ on $X$.   We let $\bm := (m_1,
\ldots , m_l)$ be an $l$-tuple of elements (not necessarily of finite order) 
which lie in a compact subgroup of $K \subset G$ and 
satisfy $m_1 \ldots m_l =1$.  
Let $X^{{\bm}} = \bigcap_{i=1}^l X^{m_i}$ 
be the fixed locus of the $m_i$.  We let $\Phi(\bm)\in\G{l}$ be the diagonal conjugacy class of the tuple $\bm$.\\

In this section, we use the logarithmic trace to construct a \emph{twisted
  pullback map} map $K_G(X) \to K_G( \Imult{l}(\Phi(\bm))$.
The twisted pullback map takes non-negative 
elements to non-negative elements and depends only on the
conjugacy class $\Phi(\bm)$

In Section~\ref{sec.orbifoldproduct} we apply this construction
when $G$ acts with finite stabilizer to define a twisted pullback map
$K_G(X) \to K(\II_G(X))$. The Euler class of the twisted
pullback of the tangent bundle ${\T}$, corresponding to the
tangent bundle of $[X/G]$, produces a class 
$c \in A^*_G(\II_G(X))$ 
which satisfies the conditions necessary for the $\star_c$
product to be a commutative and associative.

\subsection{The logarithmic restriction map in equivariant
  $K$-theory} \label{subsec.logres} 
There is an action of $Z = Z_G(\bm)$  on $X^{\bm}$, and if $V$ is $G$-bundle on $X$, then $V|_{X^{\bm}}$ has a canonical $Z$-action which necessarily commutes with the action of the $m_i$
on the fibers of $V|_{X^{\bm}}$. 
\begin{lemma} \label{lem.zarclosure}
Let  $H \subset G$ 
be the Zariski closure of
the group generated by the $m_i$.
 Then $H$ also acts trivially on
$X^{{\bm}}$.
\end{lemma}
\begin{proof}
For every point $x \in X^{{\bm}}$, the isotropy group $G_x$
contains each $m_i$. Since $G_x$ is an algebraic subgroup
of $G$, it must also contain $H$. Therefore, $H$ acts trivially on
$X^{{\bm}}$.
\end{proof}
\begin{lemma} \label{lem.commute}
The subgroups $Z$ and $H$ commute.
\end{lemma}
\begin{proof}
The commutator map $Z \times H \to G$, $(z,h)\mapsto [z,h]$
is constant on the dense subset $Z \times 
\langle m_1, \ldots m_l\rangle$, so the image of the commutator map is constant. Hence $Z$ and $H$ are commuting subgroups of $G$.
\end{proof}
As a consequence of Lemmas~\ref{lem.zarclosure} and~\ref{lem.commute}
the restriction $V|_{X^{{\bm}}}$ has a natural $Z \times H$-module structure. 
\begin{defi}
Given a $G$-equivariant vector bundle $V$ define the \emph{logarithmic restriction} of $V$ to be the class $K_Z(X^{{\bm}})$ defined by the formula
\begin{equation}  \label{eq.logrestr}
V({\bm}) = \sum_{i=1}^l
L(m_i)(V|_{X^{\bm}}) + V^{{\bm}} - V|_{X^{\bm}}.
\end{equation}

By Proposition~\ref{prop.logtracekt} and Remark~\ref{rem.canonical} $V({\bm})$  is a non-negative integral element
in $K_Z(X^{{\bm}})$.  
The assignment $V \mapsto V({\bm})$ extends linearly to give a map
$K_G(X) \to K_Z(X^{\bm})$ taking non-negative elements to non-negative elements. We refer to this map as the \emph{logarithmic restriction map.}
\end{defi}

\subsection{The twisted pullback in equivariant $K$-theory} \label{subsec.twistpullback}
The map $G \times X^{\bm } \to \Imult{l}(\Phi)$
given by 
$(h,y) \mapsto (h y, h m_1h^{-1}, \ldots,  h m_l h^{-1}) $
gives an identification $\Imult{l}(\Phi) = G \times_Z X^{{\bm}}$.  Hence by Morita
equivalence, we may identify $K_Z(X^{{\bm}})$ with
$K_G(\Imult{l}(\Phi))$. Let $V(\Phi) \in K_G(\Imult{l}(\Phi))$ be
the class identified with $V({\bm})$.
\begin{lemma} \label{lem.indep}
The class $V(\Phi) \in K_G(\Imult{l}(\Phi))$ is independent of the choice of
representative
$(m_1, \ldots , m_l) \in \Phi$.
\end{lemma}
\begin{proof}
Let
$${\bg} = (g_1, \ldots , g_l) = (h^{-1}m_1h, h^{-1}m_2 h, \ldots, h^{-1}m_l h)$$
be another $l$-tuple in $\Phi$. Then 
$h X^{\bg} = X^{\bm}$ and $h Z_{{\bg}} h^{-1} = Z_{{\bm}}$.  
Let $p\colon \Imult{l}(\Phi)
\to X$ be the map induced by the projection 
$G^l \times X \to X$. 
By
Lemma~\ref{lem.morequivincl}, the bundle $p^*V$ on $\Imult{l}(\Phi)$ is
equivalent to the bundle $V|_{X^{\bm}}$ (resp. $V|_{X^{\mathbf
g}}$ ) under the Morita equivalences between the category of $G$-modules on 
$\Imult{l}(\Phi)$ and the category of $Z_{\bm}$-modules (resp. $Z_{\bg}$-modules)
on $X^{\bm}$ (resp. $X^{\bg}$). 

Since $g_i$ is conjugate to $m_i$, the eigenvalues for the action of
$g_i$ on $V|_{X^{{\bg}}}$ are the same as the eigenvalues for the
action of $m_i$ on $V|_{X^{{\bm}}}$. If $\alpha$ is an
eigenvalue, let $V_{\alpha,m_i}$ be the $\alpha$-eigenbundle
of $V|_{X^{{\bm}}}$ and $V_{\alpha,g_i}$ the
$\alpha$-eigenbundle of $V|_{X^{\bg}}$. To complete the
proof it suffices to show that $V_{\alpha,m_i}$ and $V_{\alpha,
g_i}$ identify with the same $G$-bundle on $\Imult{l}(\Phi)$.

This can be seen as follows.
Consider the commutative triangle
$$\begin{array}{ccc}
G \times X^{{\bg}} &  \stackrel{k_g} \to& G \times X^{{\bm}}\\
\searrow & & \swarrow\\
 & \Imult{l}(\Phi) & 
\end{array},
$$
where the horizontal map $k_g$ is the isomorphism given by $(h,x)\mapsto (ghg^{-1},gx)$. If $W$ is a $G
\times Z_{\bm}$-module on $G \times X^{{\bm}}$ 
then the $1 \times Z_{{\bm}}$-invariant subsheaf of $W$ is the same as the $1 \times Z_{{\bg}}$-invariant subsheaf of $k^*W$. Since 
$V_{\alpha,g_i} = k_g^*V_{\alpha,m_i}$ the lemma follows.
\end{proof}

As an immediate consequence of Lemma~\ref{prop.logtracepositive} 
and the definition we obtain the following proposition.
\begin{prop} \label{prop.twistedpullbacK}
The map $K_G(X) \to K_G(\Imult{l}(\Phi))$, $V\mapsto V(\Phi)$ takes non-negative elements to non-negative elements.
\end{prop}

\begin{defi}
We refer to the map $K_G(X) \to K_G(\Imult{l}(\Phi))$ as the \emph{twisted
  pullback map}. 

The twisted pullback $K_G(X) \to
K_G(\Imult{l}(\Phi))$ is the composition of the logarithmic restriction
$K_G(X) \to K_Z(X^{\bm})$ with the Morita equivalence
identification $K_Z(X^{\bm}) = K_G(\Imult{l}(\Phi))$.
\end{defi}

\begin{remark} \label{rem.canononical2} As noted in Remark
\ref{rem.canonical}, if $V$ is a $G$-equivariant vector bundle on $X$, then
there is a canonical choice of a $Z$-bundle on $X^{{\bm}}$
whose class represents $\sum_{i=1}^l L(m_i)(V|_{X^{{\bm}}})$.
As a result there is a also a canonical choice of representative for
the class $V({{\bm}})$ (and hence $V(\Phi)$). 

More precisely,
if
$V|_{X^{{\bm}}}$ decomposes as 
$(V^{{\bm}} \otimes {\mathbf 1}) \oplus \bigoplus_{E} V_E \otimes E$ 
where ${\mathbf 1}$ is the
trivial representation of the group $H =\langle {\bm} \rangle$ and the
direct sum is over all non-trivial irreducible representations, then
$V(\bm)$ 
is represented by the bundle 
$\bigoplus_E V_E^{\oplus r_E}$ 
where 
$r_E$ is the non-negative integer
$\sum_{i=1}^l (L(m_i)(E)-\dim E)$
and the sum is over all non-trivial
irreducible representations of $H$. We will also use the notation
$V({\bm})$ (resp. $V(\Phi)$) to refer to the corresponding
bundles on $X^{\bm}$ (resp. $\Imult{l}(\Phi)$).
\end{remark}

\subsection{
Identities
for logarithmic restrictions}
The following 
identities
will be helpful in proving the associativity of
the 
products we are most interested in, namely, 
the $\star_{c_\T}$ and $\star_{\kclass_\T}$ products. 

Let $m_1, m_2, m_3$ be elements of $G$ which lie in
a common compact subgroup $K$ and let $m_4 = (m_1 m_2 m_3)^{-1}$.  
Consider the tuples ${\bm_{1,2}}= (m_1,m_2,(m_1m_2)^{-1})
$, 
${\bm_{2,3}} = (m_2,m_3, 
(m_2m_3)^{-1})$, 
${\bm_{12,3}} = (m_1m_2, m_3, m_4)$,
and ${\bm_{1,23}} = (m_1, m_2m_3, m_4)$. Each of
these tuples lies in $K$, so the logarithmic restriction is defined
for each tuple.  Let ${\bm} = (m_1, m_2, m_3,
m_4)$. There is a natural inclusion of $X^{\bm}$
into the fixed locus of each of the tuples defined above. This
inclusion induces restriction maps in equivariant $K$-theory
$K_{Z_{i,j}}(X^{{\bm_{i,j}}}) \to K_Z(X^{\bm})$, etc., where
$Z_{i,j} = Z_G(m_i,m_j)$. We will abuse notation and indicate the
restriction of a class in $K_{Z_{i,j}}(X^{{\bm}_{i,j}})$
to
$K_Z(X^{\bm})$ 
by
the same name.
\begin{lemma} \label{lem.v123}
Let $V$ be a $G$-bundle on $X$. 
Then following identities hold in $K_Z(X^{\bm})$.
\begin{eqnarray} 
  V({\bm_{1,2}}) + V({\bm_{12,3}}) 
& = & 
\sum_{i=1}^4 L(m_i)(V) + V^{{\bm}_{1,2}} 
+ V^{{\bm}_{12,3}} - V^{m_1m_2} - V \label{eq.v12123} \\ 
  V({\bm_{2,3}}) + V({\bm_{1,23}}) 
& = & 
\sum_{i=1}^4 L(m_i)(V) + V^{{\bm}_{2,3}} 
+ V^{{\bm}_{1,23}} - V^{m_2m_3} - V \label{eq.v23123}
\end{eqnarray}
Similarly, for all $r\geq 4$, $\bm := (m_1,\ldots,m_r)$ in $K^r$ such that $\prod_{i=1}^{r}  m_i = 1$, then for all $j=2,\ldots, r-2$, set $m := \prod_{i=1}^j m_i$.
We have the equality
\begin{equation}\label{eq:SumOfVs}
V(\bm) = V(m_1,\ldots,m_j, m^{-1}) + V(m,m_{j+1},\ldots,m_{r}) + \ce_j(V)(\bm)
\end{equation}
in $K_{Z_G(\bm)}(X^\bm)$, where 
\begin{equation}\label{eq:ExcessV}
\ce_j(V)(\bm) := V^\bm - V^{(m_1,\ldots,m_j,m^{-1})} - V^{(m,m_{j+1},\ldots,m_{r})} + V^m.
\end{equation}
\end{lemma}
\begin{proof}
The proof follows from the definition of the restricted pullback in terms of the logarithmic trace and the identity
$L(g)(V)+ L(g^{-1})(V) = V - V^g$.
\end{proof}

\section{The twisted product} \label{sec.orbifoldproduct} 

\textbf{Background hypotheses.} In Section \ref{subsec.twistedpullback} we assume that all group
actions are quasi-free. In Section \ref{subsec.twistedproduct} we assume that all group actions have finite stabilizer.
When working with equivariant Chow groups we  assume that $X$ is a smooth algebraic space and when working with equivariant 
$K$-theory we assume that $X$ is a smooth scheme.\\

In this section we define, for each positive element $V \in K_G(X)$, 
twisted products $\star_{c_V}$ and $\star_{\kclass_V}$ on
$A^*_G(I_G(X))$ and $K_G(I_G(X))$, respectively. In general the
product will be commutative but not associative. When $V = \T$ is the
element of $K_G(X)$ corresponding to the tangent bundle of the
quotient stack $[X/G]$, then the product is associative.

\subsection{The twisted pullback 
$K_G(X) \to K_G(\II_G(X))$
} \label{subsec.twistedpullback}
Let $G$ be an algebraic group acting 
quasi-freely
on an
algebraic space $X$. The construction of Section~\ref{sec.logres}
allows us to define 
twisted pullbacks
$K_G(X) \to K_G(\II_G(X))$ and, more generally, $K_G(X) \to K_G(\Imult{l}_G(X))$.
\begin{defi} \label{def.twistshift}
If $V$ is an equivariant vector bundle on $X$ and 
$\Phi = \Phi(m_1, \ldots,  m_l) \in \G{l}$
define $V^{tw}(\Phi) \in K_G(\Imult{l}(X))$ to be the class identified with $V(\Phi(m_1, \ldots , m_l, (m_1\ldots m_l)^{-1}))$ under the isomorphism
$$ \Imult{l+1}(\Phi(m_1, \ldots , m_l, (m_1 \ldots m_l)^{-1}))
\stackrel{\esub{l+1}} \to \Imult{l}(\Phi(m_1, \ldots , m_l))$$
of Proposition~\ref{prop:IsomOfInertias}.
\end{defi}

\begin{defi} \label{def.twistedpullback}
Define a twisted pullback $K_G(X) \to K_G(\Itwo_G(X))$ (and, more generally, $K_G(X) \to K_G(\Imult{l}_G(X))$) taking 
$V \mapsto V^{tw}$ by setting
$V^{tw}|_{\Itwo(\Phi)} = V^{tw}(\Phi)$
for each diagonal conjugacy class $\Phi =\Phi(m_1,m_2)\in\G{2}$ (and, more generally, $\Phi(m_1, \ldots , m_l) \in \G{l}$) 
such that
$\Itwo(\Phi) \neq \emptyset$ (and, more generally, $\Imult{l}(\Phi) \neq \emptyset$). 
\end{defi}
(Note that a necessary condition for
$\Itwo(\Phi)$ to be non-empty is that $m_1$ and $m_2$ generate a finite---hence unitarizable---subgroup $H \subset G$.)
The crucial fact about the twisted pullback is that it does not depend on the presentation for $[X/G]$ as a quotient stack.
\begin{thm} \label{thm.twcanonical}
If $V$ is a $G$-equivariant vector bundle on $X$, then the bundle $V^{tw}$
defined above is independent of the choice 
presentation of $\X =[X/G]$
as a quotient stack.
\end{thm}
\begin{proof}
Suppose that $\X = [X/G]$ is equivalent to the quotient
$[Y/H]$ for some group $H$. Let $W = X \times_\X Y$. The
two identifications of $\X$ as a quotient stack imply that
there are commuting free actions of $G$ and $H$ on $W$ such that
$W/H = X$ and $W/G = Y$ and 
$[W/(G \times H)] = \X$.
To prove the theorem it suffices to show that $V^{tw}$ as computed on $\II_G(X)$ may be identified with $V^{tw}$ as computed on $\II_{G \times H}(W)$ under the common identification with $\II_\X$.

The projection map $p \colon W \to X$ induces a morphism (which we
also call $p$) $\II_{G \times H}(W) \to \II_{G}(X)$.
Let $V^{tw,G}$ be the bundle on $\II_G(X)$ obtain by the
logarithmic twist for the $G$ action on $X$ and let $V^{tw,G \times
  H}$ be the $G \times H$-equivariant bundle on $\II_{G \times
  H}(W)$. To identify the two twists, we must show that $p^*(V^{tw,G}) =
V^{tw,G \times H}$.

Let $\Phi \in\G{2}$ be a diagonal conjugacy class, then
$p^{-1}(\Itwo(\Phi)) = \coprod_{\Phi'} \Itwo(\Phi')$, where the sum is
over the finitely many diagonal conjugacy classes $\Phi' 
\subset (G\times H \times G\times H)$ whose image in $G \times G$ is
$\Phi$ and for which $\Itwo(\Phi')\neq \emptyset$. 
Let $\Phi'$ be a conjugacy class in $G \times H \times G \times H$ whose image in $G \times G$ is $\Phi$. To prove the Theorem,
we must show that $p^*(V^{tw}(\Phi))|_{\Itwo(\Phi')} = p^*V^{tw}(\Phi')$ for all such $\Phi$ and  $\Phi'$.

Given a representative $(g_1 \times h_1, g_2 \times h_2) \in \Phi'$,
the map $p \colon W \to X$ restricts to a map $p' \colon W^{g_1 \times
  h_1, g_2 \times h_2} \to X^{g_1, g_2}$. There is a free action of $1
\times Z_H(h_1,h_2)$ on $W^{g_1\times h_1, g_2 \times h_2}$, and the
map $p'$ is a torsor for this group. Let 
$K' = \langle g_1 \times h_1, g_2 \times h_2 \rangle$. 
If $V$ is $G$-equivariant vector bundle on $X$
then there is a natural isomorphism of 
$Z_{G\times H}(g_1 \times h_1, g_2 \times h_2) \times K'$-equivariant
vector bundles.
$$p'^*(V|_{X^{g_1,g_2}}) \to (p^*V)|_{W^{g_1 \times h_1, g_2 \times h_2}}$$
\begin{lemma} \label{lem.stabs}
If $p(w) = x$, then the projection $G \times H \to G$ induces an isomorphism 
$\Stab_{G \times H}(w) \to \Stab_G(x)$.
\end{lemma}
\begin{proof}[Proof of Lemma~\ref{lem.stabs}]
If $(g,h)w = w$, then, since $p$ is $1 \times H$-equivariant map, 
$gp(w) = p(w)$. Thus the projection $G \times H \to G$ restricts to
a map $\Stab_{G \times H}(w) \to \Stab_G(x)$. Since the normal
subgroup $1 \times H \subset G \times H$ acts freely on $W$ with
quotient $X = W/(1 \times H)$, the fibers of $p$ are $(1 \times
H)$-orbits. Hence $gx = x$ if and only if $(g \times 1)w = (1 \times h)^{-1}w$
for a unique element $1 \times h \in 1 \times H$. Thus the map
$\Stab_{G \times H}(w) \to \Stab_G(x)$ is bijective.
\end{proof}
By Lemma~\ref{lem.stabs}, we see that if $g_1 \times h_1, g_2 \times h_2 \in \Stab_{G\times H}(w)$, then the projection induces an isomorphism of groups 
$$K' = \langle g_1 \times h_1, g_2 \times h_2 \rangle \to K = \langle g_1, g_2 \rangle.$$
Now if $V|_{X^{g_1, g_2}}$ decomposes as a sum $\bigoplus_E V_E \otimes E$, with $V_E$ a $Z_G(g_1,g_2)$-equivariant bundle 
and 
with the sum running over irreducible $K$-modules $E$,
then $p'^*V|_{X^{g_1,g_2}} = \bigoplus_E p'^*V_E \otimes E$. 
Thus 
$$p'^* V(g_1,g_2) = \oplus p'^* (V_E^{\oplus a(E)}) = p^*V(g_1 \times h_1, g_2 \times h_2),$$ where
$a(E)= L(g_1)(E) + L(g_2)(E) + L((g_1g_2)^{-1})(E)$. This proves
our theorem.
\end{proof} 
\subsection{The twisted product} \label{subsec.twistedproduct}
Given a positive element $V$ in 
$K_G(X)$,
we may formally define twisted products $\star_{c_V}$ and
$\star_{\kclass_V}$, where 
$c_V = \euler(V^{tw})$ 
and 
$\kclass_V =\lambda_{-1}((V^{tw})^*)$
on $A^*_G(I_G(X))= A^*(\IX)$ and $K_G(I_G(X)) = K(\IX)$, respectively. 

\begin{prop} \label{prop.twistedprodidentity}
For any positive element $V$ in $K_G(X)$, the 
$\star_{c_V}$ product 
on $A^*_G(I_G(X))$
is commutative with identity $[I(\{1\})]$
and 
the pairing $\eta$ on $A^*_G(I_G(X))\otimes\Q$ 
satisfies the Frobenius property when $[X/G]$ is complete.
Similarly, the $\star_{\kclass_V}$  product
on $K_G(I_G(X))$
is commutative with identity $[{\mathscr O}_{I(\{1\})}]$
and 
the pairing $\eta$ on $K_G(I_G(X))$
satisfies the Frobenius property when $[X/G]$ is complete.
\end{prop}
\begin{proof}
From the definition we see that if $\Phi$ is the conjugacy class of the pair 
$(1,m)$, then $V((1,m,m^{-1}))= 0$, so $\euler(V(\Phi)) = [\Itwo(\Phi)]$.
Similarly,
$\lambda_{-1}((V(\Phi))^*) = [{\mathscr O_{\Itwo(\Phi)}}]$. Hence by Propositions
\ref{prop.chowidentity} and~\ref{prop.kidentity}, the $\star_{c_V}$ and $\star_{\kclass_V}$ products have an identity.

Also, given $m_1, m_2 \in G$, $i^* V(\Phi(m_1, m_2)) =
V(\Phi(m_2,m_1))$, where $i$ is the involution on $\II_G(X)$
that switches the factors. Hence by Proposition~\ref{prop.commutative}
the $\star_{c_V}$ and $\star_{c_\kclass}$ 
products are commutative.
It is immediate from the definition that 
$\tau^* c_V = c_V$
and $\tau^* \kclass_V= \kclass_V$, so by Proposition~\ref{prop:frob} the $\star_{c_V}$ and $\star_{c_\kclass}$  
products are Frobenius when $[X/G]$ is complete.
\end{proof}
For general $V$ the $\star_{c_V}$ product (resp. $\star_{\kclass_V}$ product) will
not be associative.
\begin{lemma} \label{lem.tanbundle}
Let $\X = [X/G]$ be a DM
quotient stack, and let $p\colon X \to
[X/G]$
be the universal $G$-torsor. Then $p^*T_\X = T_X -
{\mathfrak g}$ in $K_G(X)$, where ${\mathfrak g}$ is the Lie algebra
of $G$.
\end{lemma}
\begin{proof}
The map $X \to [X/G]$ is a representable $G$-torsor, so the result follows from
the well known fact that if $P \stackrel{p} \to Y$ is $G$-torsor, then
$T_p = {\mathfrak g}$. For a proof see \cite[Lemma A.1]{EdGr:05}.
\end{proof}

\begin{defi}
We define $\T := T_X - {\mathfrak g}$ in $K_G(X)$.
\end{defi}  
By Lemma~\ref{lem.tanbundle}, $\T$ is a positive element of $K_G(X)$
and its image in $K(\X)$ is independent of the presentation of
${\mathscr X}$ as a quotient stack.
\begin{defi} The \emph{rational grading on $A^*_G(I_G(X))$} assigns
the rational number $|v| := k + a_{\T}(g)(v)$ to each homogeneous
element $v$ in $A^k_{Z_G(g)}(X^g)\subseteq A^*_G(I_G(X))$ where
$a_{\T}(g)(v)$ denotes the age of $g$ on $\T$, as given in
Definition~\ref{def:age}, evaluated on the support of $v$, and $\T$
is understood to be restricted to $X^g$.
\end{defi}
The next proposition shows that when $[X/G]$ is complete 
the rational grading is natural with respect to the pairing $\eta$
on $A^*_G(I_G(X))\otimes \Q$.
\begin{prop} 
If $[X/G]$ is complete and if $\Phi_1, \Phi_2$ are conjugacy classes
in $G$ and $v_i \in A^*_G(I(\Phi_i))\otimes\Q \subset A^*_G(I_G(X))
\otimes\Q
$
(resp. $v_i \in K_G(I(\Phi_i))
\subset K_G(I_G(X))
$) then 
$\eta(v_1,v_2) = 0$ unless
$\Phi_2 = (\Phi_1)^{-1}$ 
and $\eta(v_1,v_2) = 0$ for all homogeneous
$v_i$ in $A^*_G(I(\Phi_i))\otimes\Q$ unless $|v_1| + |v_2| = \dim
[X/G]$.
\end{prop}
\begin{proof}
The vanishing conditions on the pairing $\eta$ follow immediately from 
the definition and dimensional considerations.
\end{proof}

We now come to our main theorem. 
\begin{thm} \label{thm.assoc} 
The $\star_{c_\T}$ and
$\star_{\kclass_\T}$ products on $A^*_G(I_G(X))$ and $K_G(I_G(X))$
are associative, and the orbifold product $\star_{c_\T}$ on
$A^*_G(I_G(X))$ preserves the rational grading.

In addition the following properties hold:\\

(i) $f:X\to Y$ is an \'etale $G$-equivariant morphism then the induced
pullback $f^*:A^*_G(I_G(Y))\to A^*_G(I_G(X))$ (resp.
$f^*:K_G(I_G(Y))\to K_G(I_G(X))$) respects the $\star_{c_\T}$ product
(resp. $\star_{\kclass_\T}$ product).\\

(ii)
For all $\ell\geq 2$, we have the identity 
\begin{equation}\label{eq:MultipleProducts}
v_1\star_{c_\T} \cdots \star_{c_\T} v_\ell = \mu_*( e_1^* v_1\cdot e_2^*v_2 \cdots e_\ell^* v_\ell \cdot 
\euler(\T^{tw})) 
\end{equation}
in $A^*_G(I_G(X)),$
for all $v_i$ in $A^*_G(I_G(X))$, where $\ej:\Imult{\ell}(X) \to I_G(X)$ is the $j$-th evaluation map taking 
$(m_1,\ldots,m_\ell,x)\mapsto (m_j,x)$ 
and $\mu\colon \Imult{\ell}(X)\to I_G(X)$ takes 
$(m_1,\ldots,m_\ell,x)\mapsto (m_1 m_2\cdots m_\ell,x)$. 
The analogous identity also holds in $K_G(I_G(X))$ for 
$\star_{\kclass_\T}$.
\end{thm}

\begin{proof}[Proof of Theorem \ref{thm.assoc}]
Given $m_1, m_2, m_3$ we use the same notation as in Section 
\ref{sec.assoccriterion}
and consider the conjugacy classes of pairs $\Phi_{1,2}, \Phi_{2,3},
\Phi_{12,3}, \Phi_{12,3} \in G\times G$ as well as the conjugacy class
$\Phi_{1,2,3} \in G^3$. 

By Proposition~\ref{prop.assoc} it suffices to prove that equations
\eqref{eq.assoc2} and \eqref{eq.kassoc2} hold for the $\star_{c_{\T}}$
and $\star_{\kclass_{\T}}$ products  
for all triples 
$(m_1, m_2, m_3) \in G^3$ 
such that
$\Imult{3}(\Phi(m_1, m_2,m_3)) \neq \emptyset$. 
To do this we will 
show that the following equation holds in $K_G(
\Imult{3}(\Phi(m_1,m_2,m_3))$
\begin{equation} \label{eq.kgassoc}
e_{1,2}^* \T(\Phi_{1,2})+ \mu_{12,3}^*\T(\Phi_{12,3}) +
E_{1,2}
= e_{2,3}^* \T(\Phi_{2,3})+ \mu_{1,23}^*\T(\Phi_{1,23}) + E_{2,3}
\end{equation}
Let ${\bm} = (m_1, m_2,m_3, (m_1m_2m_3)^{-1})$ and let $Z =
Z_G(m_1,m_2,m_3)$. As usual the map 
$G \times X^{{\bm}} \to \Imult{3}(\Phi_{1,2,3})$, 
$(g,x) \mapsto (gx, gm_1g^{-1}, gm_2g^{-1},
gm_3g^{-1})$ identifies 
$\Imult{3}(\Phi_{1,2,3})$ 
with the quotient $G
\times_Z X^{{\bm}}$. 
Let $Z_{i,j} = Z_G({\bm}_{i,j})$ and let ${\mathfrak z}_{i,j} = \Lie(Z_{i,j})$,
the Lie algebra of $Z_{i,j}$.
Finally, let  ${\mathfrak z}_{k}$ be the Lie algebra of $Z_G(m_k)$ for $k = 12, 23$.
\begin{lemma} \label{lem.excess}
Under the Morita equivalence identification of 
$K_G(\Imult{3}(\Phi_{1,2,3})) = K_Z(X^{{\bm}})$ 
the class $E_{1,2}$ is identified with
\begin{equation} \label{eq.e12}
TX^{m_1m_2}|_{X^{{\bm}}} -
TX^{{\bm_{12,3}}}|_{X^{{\bm}}} -
TX^{{\bm}_{1,2}}|_{X^{{\bm}}} + TX^{{\bm}} +{\mathfrak
  z}_{1,2} + {\mathfrak z_{12,3}} - {\mathfrak z_{12}} + {\mathfrak z}
\end{equation}
and the class $E_{2,3}$ is identified with 
\begin{equation} \label{eq.e23}
TX^{m_2m_3}|_{X^{{\bm}}} -
TX^{{\bm_{1,23}}}|_{X^{{\bm}}} -
TX^{{\bm}_{2,3}}|_{X^{{\bm}}} + TX^{{\bm}} +{\mathfrak
  z}_{2,3} + {\mathfrak z_{1,23}} - {\mathfrak z_{23}} + {\mathfrak z}
\end{equation}
\end{lemma}
\begin{proof}[Proof of Lemma~\ref{lem.excess}]
We only prove the identity \eqref{eq.e12}, as the proof of
\eqref{eq.e23} is identical. By definition 
\begin{equation} \label{eq.normals}
\begin{array}{ccl}
E_{1,2}  
&= & 
\mu_{12,3}^*N_{e_1}
 - N_{e_{1,2}}\\
 & = & 
\mu_{12,3}^*e_1^*TI(\Phi_{12}) 
- \mu_{12,3}^*T\Itwo(\Phi_{12,3}) -
e_{1,2}^*T\Itwo(\Phi_{1,2})+T\Imult{3}(\Phi_{1,2,3})
\end{array}
\end{equation}

If $\Phi$ is the conjugacy class of an $l$-tuple ${\bg} = (g_1,
g_2, \ldots g_l)$, then 
$\Imult{l}(\Phi) \cong G \times_{Z_{{\bg}}} X^{{\bg}}$. 
Hence, under the Morita equivalence identification
of $K_G(\Imult{l}(\Phi))$ with 
$K_{Z_G(\bg)}( X^{\bg})$
 we
have that
\begin{equation} \label{eq.tsphi}
T\Imult{l}(\Phi) = TX^{{\bg}} - \Lie(Z_G(\bg))
\end{equation}
Substituting \eqref{eq.tsphi} into \eqref{eq.normals} yields the identity 
of \eqref{eq.e12}, provided that the various composition 
of pullbacks correspond to the restriction along the inclusions of 
the various fixed loci. This follows from Lemma~\ref{lem.morequivincl} 
\end{proof}

Now, for all $\bg = (g_1, \ldots , g_l)$ in $G^l$, we have $(TX|_{X^{\bg}})^{\bg} = TX^{\bg}$. Likewise, ${\mathfrak g}^{\bg}= {\mathfrak z_{\bg}}$, where ${\mathfrak z_{\bg}} =
\Lie(Z_G(\bg))$. 

Combining the formulas of Lemma~\ref{lem.excess} with
equations \eqref{eq.v12123} and \eqref{eq.v23123}  
with $V = {\T} = TX - {\mathfrak g}$ 
yields, 
\begin{eqnarray}\label{eq:BaseOne}
\T(m_1,m_2,m_3,m_4) &=& \T({\mathbf  m}_{1,2}) + \T({\bm}_{12,3}) +E_{1,2}  \\
\label{eq:BaseTwo}&=&  \T({\mathbf  m}_{2,3}) + \T({\bm}_{1,23}) +E_{2,3}.
\end{eqnarray}
Therefore the $\star_{c_\T}$ and $\star_{\kclass_\T}$ products are associative.

The fact that $\star_{c_\T}$ preserves the rational grading follows immediately from Equation (\ref{eq.logrestr}), where $V =\T$.

If $f:X\to Y$ is an \'etale $G$-equivariant morphism then the induced pullback $f^*:A_G^*(I_G(Y))\to A_G^*(I_G(X))$ respects the orbifold products since $f^*(TY-\mathfrak{g}) = TX - \mathfrak{g}$.  The argument is identical for K-theory.

To prove Equation (\ref{eq:MultipleProducts}), we first observe that
the $\ell = 3$ case follows from either Equations (\ref{eq:BaseOne})
or (\ref{eq:BaseTwo}).  The general case follows by induction. Suppose
that Equation (\ref{eq:MultipleProducts}) holds for $\ell$ then write
\[
v_1\star_{c_\T} \cdots \star_{c_\T} v_\ell\star_{c_\T} v_{\ell+1} = (v_1\star_{c_\T} \cdots \star_{c_\T} v_\ell)\star_{c_\T} v_{\ell+1}
\]
and apply the induction hypothesis to the expression in the parenthesis, using  excess intersection theory and Equation (\ref{eq:SumOfVs}) for $r = \ell+2$ and $j=\ell$ to obtain the desired result.
\end{proof}

\begin{remark}
When $G$ is finite, then the character formula of \cite[Lemma
8.5]{JKK:07} implies that $\T^{tw}$ equals the class of the
obstruction bundle defined in the paper of Fantechi-G\"ottsche
\cite{FaGo:03}. Hence, the 
$\star_{c_\T}$ 
product on $A^*_G(X)\otimes \Q = 
A^*(X)^G$ equals their product.
\end{remark}

\begin{remark}
The analogue of Equation (\ref{eq:MultipleProducts}) in the context of torus actions on symplectic manifolds was proven in \cite{GHK:07}. 
\end{remark}

\begin{example}\label{exam.lusztig}
Let $G$ be a finite group acting upon a point $X$. In this case,
$I_G(X) = G$ where $G$ acts upon itself by conjugation. Therefore,
the orbifold $K$-theory $K_G(I_G(X))$ is additively equal to
$K_G(G)$, the Grothendieck group of $G$-equivariant vector bundles
on $G$, but the orbifold product is not the usual product on
equivariant K-theory. Since the tangent bundle to $G$ has rank zero,
$\T^{tw}$ vanishes and the orbifold product of two $G$-equivariant
vector bundles $V$ and $W$ over $G$ is given by
\[
V \star W = \mu_* (e_1^* V\otimes e_2^* W)
\]
where $e_i : G\times G\to G$ are the projections onto the $i$-th
factor and $\mu:G\times G\to G$ is the multiplication.  The $\star$
product on $K_G(G)$ coincides with a product introduced by Lusztig 
\cite{Lu:87}. Kaufmann and Pham \cite[Theorem 3.13]{KaPh:07} show that
this product also coincides with a product on the representation ring
of the Drinfeld double of $G$ which appears in the physics literature
(cf. \cite{DPR:90}).
\end{example}

\begin{example} \label{exam.otherassoc}
If we set $c_{\Phi(g,1)}= [\Itwo(\Phi(g,1))]$ and $c_{\Phi(1,g)} =
[\Itwo(\Phi(g,1))]$ and $c(\Phi) = 0$ for all other conjugacy classes in
$\G2$, then the $\star_c$  
product is commutative and
associative. This is the restriction of the $\star_{c_\T}$ product
obtained by setting all products $\alpha_{\Psi_1} \star
\alpha_{\Psi_2}$ to be 0 unless $\Psi_1$ or $\Psi_2$ is the
conjugacy class $\{1\}$. If $g \in \Psi$ is a representative
element, and we identify 
$A^*_G(I(\Psi)) = A^*_{Z_G(g)}( X^g)$ 
then we
have a simple formula for the $\star_c$ product:
\begin{equation}
\alpha_{\{1\}} \star_c \alpha_\Psi = j^*\alpha_{\{1\}} \cdot \alpha_g,
\end{equation}
where $j^*\colon A^*_G(X)= A^*_G(I(\{1\}) \to A^*_{Z_G(g)}(X^g)$ is the
composition of the restriction of groups map $A^*_G(X) \to
A^*_{Z_G(g)}(X)$ with pullback in $Z_G(g)$-equivariant Chow groups for
the regular embedding $X^g \hookrightarrow X$, and $\alpha_g \in
A^*_{Z_G(g)}(X^g)
$ is the class which is Morita equivalent to $\alpha_{\Psi}
\in A^*_G(I(\Psi))$.
\end{example}

\section{Orbifold Riemann-Roch}

\textbf{ Background hypotheses.} In this section, all spaces are required
to be schemes and all group actions have finite stabilizer.\\

In this section, we extend the orbifold Riemann-Roch theorem of \cite{JKK:07}
to quotient stacks $\X = [X/G]$ with $G$ an arbitrary linear
algebraic group acting with finite stabilizer on a smooth scheme $X$. In
particular, we show that a generalization of the stringy Chern
character defines a ring homomorphism 
$$\cho \colon K_G(I_G(X))  \to A^*_G(I_G(X))\otimes\Q$$ 
of orbifold ($\star_{\kclass_\T}$ and $\star_{c_\T}$) products. 
Note, however, that (when $X$ is complete) $\cho$ does not preserve the 
pairing.
Applying the equivariant
Riemann-Roch theorem of \cite{EdGr:00}, the map $\cho$ factors through
an isomorphism
 $K_G(I_G(X))_1 \to A^*_G(I_G(X))\otimes\Q$, 
where $K_G(I_G(X))_{1}$ is a
distinguished summand in $K_G(I_G(X))\otimes \Q$ which generalizes the small
orbifold $K$-theory of \cite{JKK:07} for quotients by a finite group.

After tensoring with $\C$, the summand 
$K_G(I_G(X))_1$ may be identified via the
Riemann-Roch isomorphism of \cite{EdGr:00} with $K_G(X)\otimes \C =
K([X/G])\otimes \C$. As a corollary we obtain a orbifold product on
$K(\X)\otimes \C$. We give an explicit description of this product
in Section~\ref{sec.orbonktheory}.

\subsection{Background on equivariant Riemann-Roch}
We recall some basic facts that were proved in \cite{EdGr:00} and
\cite{EdGr:05} about the equivariant
Riemann-Roch theorem. 
Suppose that $Y$ is a smooth algebraic space on which
an algebraic group $G$ acts. The Chern character defines a ring
homomorphism $\ch\colon K_G(Y) \to \prod_{i =0}^\infty A^i_G(Y)\otimes \Q$.
When the group acts 
quasi-freely, 
then by
\cite{EdGr:98,EdGr:00}, $A^i_G(Y) \otimes \Q= 0$ for $i$ sufficiently large, so
the infinite direct product 
may be identified with the usual rational equivariant Chow
ring.  If the quotient stack $[Y/G]$ satisfies the resolution property
(in particular if $Y$ is a smooth scheme), then we may identify the
Grothendieck group of equivariant vector bundles $K_G(Y)$ with the
Grothendieck group of $G$-linearized coherent sheaves.

In
this case, the equivariant Riemann-Roch theorem of \cite{EdGr:00}
implies that the map $\ch \colon K_G(Y) \to \prod_{i =0}^\infty
A^i_G(Y)\otimes \Q$ factors through an isomorphism $\widehat{K_G(Y)} \to
\prod_{i =0}^\infty A^i_G(Y)$, where $\widehat{K_G(Y)}$ 
is the
completion of $K_G(X)\otimes \Q$ at 
the augmentation ideal of $\Rep(G)\otimes \Q$.  
When $G$ acts 
quasi-freely
then $K_G(Y)$ is supported at a finite number of
points of the representation ring $\Rep(G)$, 
and we may identify the
augmentation completion $\widehat{K_G(Y)}$ with the localization at
the same ideal. This localization, which we
denote by $K_G(Y)_1$, is a summand in 
$K_G(Y)\otimes \Q$. 
If we let ${\mathscr Y} = [Y/G]$ be the
quotient stack, then $K_G(Y)_1$ may be identified with the
completion of $K({\mathscr Y})\otimes \Q$ 
at its augmentation ideal. In
particular, it is independent of the choice of presentation for
${\mathscr Y}$ as a quotient stack.

\begin{thm} (cf. \cite{EdGr:00})
The equivariant Chern character
homomorphism
$\ch \colon K_G(Y) \to A^*_G(Y)\otimes \Q$
factors through an isomorphism $K_G(Y)_1 \to A^*_G(Y)\otimes \Q$.
\end{thm}
\begin{proof}
Since $Y$ is smooth, 
\cite[Theorem 3.3]{EdGr:00} states that the map
$$\tau^G_X \colon K_G(Y)_1 \to A^*_G(Y), \text{ given by } V \mapsto \ch(V)
\Td(\T)$$ is an isomorphism. However, the class 
$\Td(\T)$ is invertible in $A^*_G(Y)\otimes \Q$, 
so the Chern character
homomorphism is also an isomorphism, after completing at the augmentation ideal.
\end{proof}

\subsection{The orbifold Chern character}
As observed in \cite{JKK:07}, the Todd class map $K(X) \to A^*(X)
\otimes \Q$ can be formally extended to a map $K(X) \otimes \Q \to
A^*(X) \otimes \Q$ by defining 
$\Td({\frac{1}{n} V})$ 
to be the unique in
element of the form $t = 1 + t_1 + \ldots t_m$ with $t_i \in A^i(X)$
that satisfies the equation $t^n = \Td(V) \in A^*(X)$. The same
argument works for equivariant Chow groups and we can define
$\Td({1\over{n}}V) \in \Pi_{i=0}^\infty A^i_G(X)$ for any equivariant
vector bundle $V$.

If $\Psi$ is a conjugacy class in $G$ and $V$ is a $G$-equivariant
vector bundle on $X$, let $L(\Psi)(V) \in K_G(I_\Psi)\otimes \Q$ be
the class which is Morita equivalent to $L(m)(V|_{X^m})$ for any $m \in \Psi$.
The argument used to prove Lemma~\ref{lem.indep} shows that
$L(\Psi)(V)$ is independent of the choice of $m$.

\begin{defi}
Define the orbifold Chern character 
$$\cho \colon K_G(I_G(X)) \to A^*_G(I_G(X))\otimes \Q$$
by the formula
\begin{equation}
\cho(\Fc_\Psi) =\ch(\Fc_\Psi)
\Td(-L(\Psi)({\T}))
\end{equation}
for a class 
$\Fc_\Psi \in K_G(I(\Psi))$.
\end{defi}
We now obtain the following generalization of \cite[Theorem
6.1]{JKK:07} to groups of positive dimension.
\begin{thm} \label{thm.orbrr}
The map $\cho \colon K_G(I_G(X)) \to A^*_G(I_G(X))\otimes\Q$ is a homomorphism when $K_G(I_G(X))$ has the
$\star_{\kclass_{{\T}}}$ product and 
$A^*_G(I_G(X))\otimes\Q$
has the $\star_{c_\T}$ product. 
Moreover, the map $\cho$ factors through an isomorphism
$$\cho \colon K_G(I_G(X))_1 \to A^*_G(I_G(X))\otimes\Q.$$
\end{thm}
\begin{proof}
The proof is analogous to the proof of Theorem 6.1 in \cite{JKK:07}.
Given conjugacy classes $\Psi_1$ and $\Psi_2$ and 
$\alpha_1 \in K_G(I(\Psi_1))$, $\alpha_2 \in K_G(I(\Psi_2))$, then
by definition
\begin{equation} \label{eq.ktprop}
\cho(\alpha_1 \star_{\kclass_\T} \alpha_2) = 
\cho\left( \sum_{\Phi_{1,2}} 
\mu_*(e_1^*\alpha_1 \otimes e_2^*\alpha_2 \otimes 
\lambda_{-1}(\T(\Phi_{1,2}))^*) \right),
\end{equation}
where the sum on the right-hand side of \eqref{eq.ktprop} is over all
$\Phi_{1,2} \in \G{2}$ 
satisfying $e_1(\Phi_{1,2}) =
\Psi_1$, $e_2(\Phi_{1,2}) = \Psi_2$.

Similarly, 
\begin{equation} \label{eq.itprop}
\cho(\alpha_1) \star_{c_\T} \cho(\alpha_2) = \sum_{\Phi_{1,2}}
\mu_*( \cho(\alpha_1) \cho(\alpha_2) 
\euler(\T(\Phi_{1,2})))
\end{equation}
To prove that $\cho$ commutes with the twisted product, it
suffices to show that the right-hand sides of \eqref{eq.ktprop} and
\eqref{eq.itprop} are termwise equal. 
Consider $\Phi_{1,2} \in\G{2}$ 
satisfying $e_1(\Phi_{1,2}) = \Psi_1$, $e_2(\Phi_{1,2}) =\Psi_2$.
Let $\Psi_{12} = \mu(\Phi_{1,2})$.
By definition 
\begin{multline} \label{eq.grr}
\cho\left(\mu_*(e_1^*\alpha_1 \otimes e_2^* \alpha_2 
\otimes \lambda_{-1}((\T(\Phi_{1,2})^*) ))\right) =\\
\ch\left(\mu_*(e_1^*\alpha_1 \otimes e_2^* \alpha_2 
\otimes \lambda_{-1}((\T(\Phi_{1,2}))^*))\right)\Td(-L(\Psi_{12})(\T))
\end{multline}
By the equivariant Grothendieck-Riemann-Roch theorem, the right-hand 
side of \eqref{eq.grr} equals
\begin{equation} \label{eq.eqgrr1}
\mu_*\left[\ch(e_1^*\alpha_1 \otimes e_2^*\alpha_2 \otimes 
\lambda_{-1}((\T(\Phi_{1,2}))^*))\Td(T_\mu)\right]\Td(-L(\Psi_{12})(\T))
\end{equation}
Now if $V$ is a positive element in
equivariant $K$-theory, then 
$\ch(\lambda_{-1}(V^*))\Td(V) = \euler(V)$. 
Thus,
using the multiplicativity of the ordinary equivariant Chern character, we may rewrite the right hand side of 
\eqref{eq.eqgrr1} as
\begin{equation} \label{eq.rr2}
\mu_*\left[\ch(e_1^*\alpha_1) \ch(e_2^*\alpha_2) 
\euler(\T(\Phi_{1,2}))
\Td(-\T(\Phi_{1,2})))\Td(T_\mu)\right]\Td(-L(\Psi_{12})(\T))
\end{equation}
Applying the projection formula, \eqref{eq.rr2} can be rewritten as 
\begin{equation} \label{eq.rr3}
\mu_*\left[e_1^*\ch(\alpha_1) e_2^*\ch(\alpha_2) 
\euler(\T(\Phi_{1,2}))
\Td(T_\mu)\Td(-\T(\Phi_{1,2}))
\Td(-\mu^*L(\Psi_{12})(\T))\right]
\end{equation}
On the other hand,
\begin{equation} \label{eq.rr4}
\begin{split}
& \cho(\alpha_1) \star_{c_\T} \cho(\alpha_2) \\
& \;\;\;\;= \mu_*\left[e_1^*(\ch(\alpha_1) \Td(-L(\Psi_1)(\T)))
 e_2^*(\ch(\alpha_2) \Td(-L(\Psi_2) ({\T}))) 
\euler(\T(\Phi_{1,2}))\right]\\
&  \;\;\;\; = \mu_*\left[e_1^*\ch(\alpha_1) e_2^*\ch(\alpha_2) 
\euler(\T(\Phi_{1,2}))
\Td(-e_1^*L(\Psi_1)(\T))\Td(-e_2^*L(\Psi_2)(\T))\right] 
\end{split}
\end{equation}
Comparing \eqref{eq.rr3} and \eqref{eq.rr4}, it suffices to show that 
the equation
\begin{equation} \label{eq.ktrr}
T_\mu -\T(\Phi_{1,2}) - \mu^*L(\Psi_{12})(\T) = -e_1^*L(\Psi_1)(\T) - 
e_2^*L(\Psi_2)(\T)
\end{equation}
holds in $K_G(I(\Psi_{1,2})$.
If $(m_1, m_2) \in \Psi$, then by Morita equivalence it suffices to prove the corresponding identity in $K_{Z_{\bm}}(X^{\bm})$, where ${\bm} = (m_1, m_2, (m_1 m_2)^{-1})$. 
By definition, 
\begin{equation} \label{eq.tmu}
T_\mu = \mu^*TI(\Psi_{12}) - T\Itwo(\Phi_{1,2})
\end{equation}
As noted above, if $\Phi$ is the conjugacy class of an $l$-tuple
${\bg}= (g_1, \ldots ,g_l)$, then $T\Imult{l}(\Phi)= \T^{{\mathbf
    g}}$. Hence by Lemma~\ref{lem.morequivincl}, the right-hand side of
\eqref{eq.tmu} is Morita equivalent to
\begin{equation} \label{eq.tmu1}
\T^{m_1m_2}|_{X^{\bm}} - \T^{{\bm}}.
\end{equation}
Substituting \eqref{eq.tmu1} and 
the definition of the logarithmic trace we see that the left-hand side of 
\eqref{eq.ktrr} 
is Morita equivalent to the
class
\begin{equation} \label{eq.ktrrlhs1}
\T^{m_1 m_2}|_{X^{\bm}} - \T({\bm}) -
L(m_1m_2)(\T)|_{X^{\bm}}.
\end{equation}
Applying the identity $L(g)(V) + L(g^{-1})(V) = V - V^g$ with $g = m_1
m_2$ and $V = \T$, we can rewrite \eqref{eq.ktrrlhs1}
as 
\begin{equation} \label{eq.ktrrlhs2}
-\T({\bm}) + \T^{\bm} - \T^{m_1m_2}|_{X^{\bm}} +
L(m_1m_2)(\T)|_{X^{\bm}}.
\end{equation}
By definition 
\begin{equation} \label{eq.defoftwist}
\T({\bm}) = L(m_1)(\T)|_{X^{\bm}} +
L(m_2)(\T)|_{X^{\bm}} + L((m_1m_2)^{-1})(\T)|_{X^{\bm}} +
\T^m|_{X^{\bm}} - \T^{\bm}.
\end{equation}
Substituting \eqref{eq.defoftwist} into \eqref{eq.ktrrlhs2}, we see
that the left-hand side of \eqref{eq.ktrr} simplifies to
\begin{equation} \label{eq.ktrrlhs3}
-L(m_1)(\T)|_{X^{\bm}} - L(m_2)(\T)|_{X^{\bm_2}},
\end{equation}
but the right hand side of \eqref{eq.ktrr} is Morita equivalent to the
same class, since $L(\Psi_i)(\T)$ is Morita equivalent to $L(m_i)(\T)$. 
Therefore $\cho$ defines a ring homomorphism.

Moreover, since the class $\Td(L(\Psi))$ is invertible in
$A^*_G(I(\Psi)) \otimes \Q$
and the 
$\star_{e_\T}$ commutes with the natural $\Rep(G)$ 
action on $K_G(X)$ we see that the
map $\cho$ factors through a ring 
isomorphism of twisted products $K_G(X)_{1} \to A^*_G(I_G(X))\otimes \Q$. 
\end{proof}
\begin{remark}
The localization $K_G(I_G(X))_1$ is the analogue, for $\dim G > 0$,
of the \emph{small orbifold $K$-theory} of \cite{JKK:07}. 
In particular, 
when $G$ is a finite group
then $K_G(X)_1$ is isomorphic to 
$K_{orb}([X/G])$, the 
small orbifold $K$-theory of $[X/G]$.
This follows from the fact that both are  isomorphic to the orbifold Chow ring.
\end{remark}
\section{A twisted product on $K_G(X) \otimes \C$}
\label{sec.orbonktheory} Given a commutative, twisted product $\star$
on $K_G(I_G(X))$ which commutes with the $\Rep(G)$-algebra 
structure on
$K_G(I_G(X))$, the results of \cite{EdGr:00} and \cite{EdGr:05}
allow one to define a corresponding product on $K_G(X) \otimes \C$.
In particular, the twisted product $\star_{\kclass_\T}$ induces an
orbifold product on $K_G(X) \otimes \C$. In this case there is an
orbifold Chern character isomorphism of $K_G(X)\otimes \C$ with
$A^*_G(I_G(X))\otimes \C$. 

\begin{remark}
The isomorphism of vector spaces
$K_G(X)\otimes \C \to A^*_G(I_G(X))\otimes \C$ may be viewed as
an algebraic analogue Adem and Ruan's
isomorphism \cite{AdRu:03} between equivariant topological $K$-theory and  
equivariant cohomology of $I_G(X)$. 
However, the products that appear in the Chern character described in  \cite{AdRu:03} are not the same as ours.  In particular, the product they use on equivariant topological $K$-theory is the tensor product, and the product they use on the equivariant cohomology of $I_G(X)$ is {not} the Chen-Ruan orbifold product.
\end{remark}
\subsection{A decomposition of $K_G(X) \otimes \C$ and the non-Abelian localization theorem}
When $G$ acts 
quasi-freely
then by \cite{EdGr:00,EdGr:05,VeVi:02} the 
$\Rep(G)\otimes \C$-module $K_G(X)\otimes \C$
is supported at a finite number
of maximal ideals ${\mathfrak m}_\Psi \subset \Rep(G)\otimes \C$,
corresponding to conjugacy
classes $\Psi$ such that $I(\Psi) \neq \emptyset$. As in
\cite{EdGr:05}, we use the notation ${\mathfrak m}_\Psi$ to refer to
the maximal ideal of virtual representations whose character vanishes
on $\Psi$.  As a result, equivariant $K$-theory decomposes into a
direct sum of its localizations.
\begin{equation}
K_G(X)\otimes \C = \bigoplus_{\Psi} K_G(X)_{{\mathfrak m}_\Psi},
\end{equation}
where the sum is over the finite number of conjugacy classes $\Psi$
such that $I(\Psi)\neq \emptyset$.
\begin{remark}
If we view                $\Rep(G) \otimes \C$ as the ring of polynomial class
functions on $G$, then the component $\alpha_\Psi$ of a class $\alpha$ in
the summand $K_G(X)_{{\mathfrak m}_\Psi}$ equals $1_\Psi \alpha$,
where           $1_\Psi\in \Rep(G) \otimes \C$ is any polynomial satisfying
$1_\Psi(\Psi) = 1$ and $1_\Psi(\Psi') = 0$ for all other $\Psi' \in
\Supp(K_G(X)\otimes \C)$.  Since the support of $K_G(X)\otimes
\C$ is finite, such a function always exists. If $G$ is infinite,
there is no canonical choice for the function $1_\Psi$ because we
impose no conditions on the value away from the support of $K_G(X)
\otimes \C$. However, different choices for $1_\Psi$ yield the same
product $1_\Psi \alpha \in K_G(X) \otimes \C$.
\end{remark}

Let $f\colon
I_G(X) \to X$ be the projection.
Since $G$ is assumed to act with finite stabilizer, the map
$f$ is a finite l.c.i.~morphism.
Fix a conjugacy class $\Psi\in G$
such that $I(\Psi)\neq \emptyset$. By Proposition~\ref{prop.inertiadecomp},
$I(\Psi)$ is open and closed in $I_G(X)$, so the restriction of $f$
to $I(\Psi)$ is also finite and l.c.i.
Choose $h \in \Psi$ and identify
$I(\Psi) = G \times_Z X^h$, where $Z=Z_G(h)$ is the centralizer of $h$ in $G$.
Let ${\mathfrak m}_h \subset \Rep(Z)\otimes \C$ be the ideal of virtual
representations whose character vanishes at the 
central conjugacy class 
$\{h\} \subset Z$. 
Since $Z$ acts with finite stabilizer on
$X^h$, we have that $K_Z(X^h)_{{\mathfrak m}_h}$ is a summand in $K_Z(X^h)\otimes \C$.
Following \cite{EdGr:05}, we denote by $K_G(I(\Psi))_{{\rm cent}_\Psi}$
the summand in $K_G( I(\Psi))\otimes \C$ which is Morita equivalent to
$K_Z(X^h)_{{\mathfrak m}_h}$. By \cite[Lemma 4.6]{EdGr:05},
$K_G(I(\Psi))_{{\rm cent}_\Psi}$ is independent of the choice of
representative $h \in \Psi$. Having established the necessary
notation, the non-Abelian localization theorem is as follows.  
\begin{thm}\cite[Theorem 5.1]{EdGr:05}
The pushforward 
$f_{*}\colon K_G(I(\Psi))_{{\rm cent}_\Psi} \to K_G(X)_{{\mathfrak m}_\Psi}$
is an isomorphism. If $\alpha \in K_G(X)_{{\mathfrak m}_\Psi}$, then 
\begin{equation} \label{eq.explicitlocalization}
\alpha = f_*\left( {f^*\alpha\over{\lambda_{-1}(N_f^*)}}\right)
\end{equation}
\end{thm}  
\subsection{Multiplicative twisting in equivariant $K$-theory}
Let $Y$ be an algebraic space 
with the action of an algebraic group
$Z$ and let $V$ be a $Z$-bundle on $Y$.  Let $h$ be an automorphism of
the fibers of $V/Y$ which commutes with the action of $Z$ on $V$.  
We can consider an exponential analogue of the logarithmic trace and
define the multiplicative twist $V^{mult}(h)$ of $V$ in $K_Z(Y)
\otimes \C$ by the formula
\begin{equation} \label{eq.multtwist}
V^{mult}(h) = \sum_\chi \chi V_\chi,
\end{equation}
where the sum is over all eigenvalues $\chi$ for the action of $h$ on the fibers of $V/X$ and 
$V_\chi$ denotes the $\chi$-eigenspace. The formula \eqref{eq.multtwist} extends to an automorphism 
$$t_g \colon K_Z(X)\otimes\C \to K_Z(X)\otimes \C$$ with 
inverse $t_{g^{-1}}$.

As in \cite{EdGr:05, EdGr:08}, we will consider the twist in the following
special case. If $h \in \Psi$, then $h$ is central in $Z=Z_G(h)$ 
and acts trivially
on $X^h$. Hence there is an action of $h$ on the fibers of any
$Z$-bundle on $X^h$.  In this case, the automorphism $t_h$ induces an
isomorphism of the summands $K_Z(X^h)_{1}$ and
$K_Z(X^h)_{{\mathfrak m}_h}$. Via the usual Morita equivalence there is
  an induced isomorphism
$$t_\Psi \colon K_G(I(\Psi))_{1} \to K_G(I(\Psi))_{{\rm cent}_\Psi}$$
which is independent of the choice of representative $h \in \Psi$.
To simplify notation, let 
$$t \colon  K_G(I_G(X))_{1} \to \bigoplus_\Psi K_G(I_G(X))_{{\rm cent}_\Psi}$$
be the map whose restriction to the summand $K_G(I(\Psi))_1$ is $t_\Psi$.

In order to define the twisted product on $K_G(X)\otimes \C$, we will need to combine the non-Abelian localization map with the multiplicative twist.
\begin{defi}
Let $f^! \colon K_G(X)\otimes \C \to K_G(I_G(X))_{1}$ given by the formula
\begin{equation}
f^!\alpha_\Psi = t^{-1} \left({f^*\alpha_\Psi\over{\lambda_{-1}(N_f^*)}}\right)
\end{equation}
for $\alpha_\Psi \in K_G(X)_{{\mathfrak m}_\Psi}$.
\end{defi}
\begin{prop}
The map $f^!$ is an isomorphism $K_G(X) \to K_G(I_G(X))_{1}$ and
$(f^!)^{-1} = f_* \circ t$
\end{prop}
\begin{proof}
The restriction of $f^!$ to $K_G(X)_{{\mathfrak m}_\Psi}$ is the
composition of isomorphisms $f^*_\Psi \colon K_G(X)_{{\mathfrak m}_\Psi} 
\to K_G(I(\Psi))_{{\rm cent}_\Psi}$ and $t^{-1}_\Psi \colon
K_G(I(\Psi))_{{\rm cent}_\Psi} \to K_G(I(\Psi))_1$, so it is an
isomorphism. Also, given $\alpha_\Psi \in K_G(X)_{{\mathfrak m}_\Psi}$,
we have
\begin{eqnarray*}
(f_*\circ t)(f^!\alpha_\Psi) & = & f_*\circ 
f^*\left({\alpha_\Psi\over{\lambda_{-1}(N_f^*)}}\right) \\
& = & \alpha_{\Psi},
\end{eqnarray*}
where the second equality follows from \eqref{eq.explicitlocalization}.
Conversely, if $\beta_\Psi \in K_G(I(\Psi))_1$, then
\begin{eqnarray*}
f^!(f_*(t(\beta_\Psi))) & = & 
t_\Psi^{-1}f^*\left(f_*\left({t(\beta_\Psi)\over{\lambda_{-1}(N_f^*)}}\right)\right)\\
& = & t^{-1}\left( \lambda_{-1}(N_f^*){t(\beta_\Psi)\over{\lambda_{-1}(N_f^*)}}\right)\\
& = & \beta_\Psi,
\end{eqnarray*}
where the second equality follows from the self-intersection formula for the map $f$.
\end{proof}

\subsection{The twisted product on $K_G(X) \otimes \C$}
We now have the necessary terminology to define the twisted product on 
$K_G(X) \otimes \C$. 
\begin{defi}
Given classes $\alpha_1, \alpha_2 \in K_G(X)\otimes \C$
set 
\begin{equation}
\alpha_1 \star_{\T} \alpha_2 = f_* t(f^!(\alpha_1) \star_{\kclass_\T} f^!(\alpha_2)))
\end{equation}
\end{defi}
The following is an immediate consequence Theorem~\ref{thm.assoc}.
\begin{thm} \label{thm.orbonkt}
The $\star_\T$ product on $K_G(X) \otimes \C$ is commutative and associative with identity
element equal to the 
the projection of $[{\mathscr O}_X]$ in $K_G(X)_1$.
\end{thm}
We also obtain as corollary of the Riemann-Roch isomorphism.
\begin{corol}
The map $\cho \circ f^! \colon K_G(X)\otimes \C \to A^*_G(I_G(X))\otimes \C$
is a ring isomorphism for the $\star_\T$ product on $K_G(X) \otimes \C$ and the $\star_{c_\T}$ product on
$A^*_G(I_G(X))$.
\end{corol}

\def\cprime{$'$}

\end{document}